\documentclass{amsart}
\usepackage{hyperref}

\usepackage{comment}
\usepackage{amsfonts, fancyhdr, qtree, amsmath, tipa, amssymb, hyperref, url, amsthm, subfigure, xy, tikz-cd, verbatim, geometry}

\usepackage[color=cyan!40]{todonotes}

\theoremstyle{definition}
\newtheorem{nul}{}[section]
\newtheorem{dfn}[nul]{Definition}

\newtheorem{rmk}[nul]{Remark}
\newtheorem{cnstr}[nul]{Construction}

\newtheorem{exm}[nul]{Example}

\newtheorem*{dfn*}{Definition}
\newtheorem*{axm*}{Axiom}
\newtheorem*{ntn*}{Notation}
\newtheorem*{exm*}{Example}
\newtheorem*{exr*}{Exercise}
\newtheorem*{int*}{Intuition}
\newtheorem*{qst*}{Question}

\theoremstyle{plain}

\newtheorem{thm}[nul]{Theorem}
\newtheorem{prop}[nul]{Proposition}
\newtheorem{lem}[nul]{Lemma}
\newtheorem{cor}[nul]{Corollary}

\newtheorem*{thm*}{Theorem}
\newtheorem*{prop*}{Proposition}
\newtheorem*{cor*}{Corollary}
\newtheorem*{lem*}{Lemma}
\newtheorem*{cnj*}{Conjecture}

\DeclareMathOperator{\Hom}{\text{Hom}}

\DeclareMathOperator{\smsh}{\wedge}

\DeclareMathOperator{\Ext}{\text{Ext}}

\DeclareMathOperator{\HFPSS}{\text{HFPSS}}

\DeclareMathOperator{\MaySS}{\text{MaySS}}
\newcommand{\F}{\mathbb{F}_2}

\newcommand{\HZ}{H\underline{\mathbb{Z}}}
\newcommand{\as}{a_\sigma}

\newcommand{\EMaySS}{C_2\text{-MaySS}}

\usepackage{tikz}
\usetikzlibrary{matrix,arrows,decorations}
\usepackage{tikz-cd}

\title{Real Orientations of Lubin--Tate spectra}

\author{Jeremy Hahn}
\email{jhahn01@mit.edu}
\address{Department of Mathematics, Massachusetts Institute of Technology, Cambridge, MA 02139}

\author{XiaoLin Danny Shi}
\email{dannyshi@math.uchicago.edu}
\address{Department of Mathematics, University of Chicago, 5734 S University Ave, Chicago, IL 60637}

\begin{document}

\begin{abstract}
We show that Lubin--Tate spectra at the prime $2$ are Real oriented and Real Landweber exact.  The proof is by application of the Goerss--Hopkins--Miller theorem to algebras with involution.  For each height $n$, we compute the entire homotopy fixed point spectral sequence for $E_n$ with its $C_2$-action given by the formal inverse.  We study, as the height varies, the Hurewicz images of the stable homotopy groups of spheres in the homotopy of these $C_2$-fixed points.
\end{abstract}


\setcounter{tocdepth}{1}
\maketitle

\tableofcontents

\section{Introduction}

The main results of this paper are the following: 

\begin{thm} \label{FullVersionONE}
Let $k$ be a perfect field of characteristic $2$, $\Gamma$ a height $n$ formal group law over $k$, and $E_{(k, \Gamma)}$ the corresponding Lubin--Tate theory.  Suppose $G$ is a finite subgroup of the Morava stabilizer group that contains the central subgroup $C_2$.  Then there is a $G$-equivariant map
$$N^{G}_{C_2} MU_\mathbb{R} \longrightarrow E_{(k, \Gamma)},$$
where $N_{C_2}^G(-)$ is the Hill--Hopkins--Ravenel norm functor. 
\end{thm}

Theorem~\ref{FullVersionONE} establishes the first known connection between the obstruction-theoretic actions on Lubin--Tate theories and the geometry of complex conjugation.  In particular, when $G = C_2$, Theorem~\ref{FullVersionONE} implies that for all height $n \geq 1$, the classical complex orientation $MU \to E_n$ can be refined to a Real orientation 
$$MU_\mathbb{R} \longrightarrow E_n.$$

This presence of geometry has tremendous computational consequences.  Using Theorem~\ref{FullVersionONE}, we obtain the first calculations for $E_n^{hC_2}$, valid for arbitrarily large heights $n$:

\begin{thm} \label{FullVersionTWO}
The $E_2$-page of the $RO(C_2)$-graded homotopy fixed point spectral sequence of $E_n$ is 
$$E_2^{s, t} (E_n^{hC_2}) = W(\mathbb{F}_{2^n})[[\bar{u}_1, \bar{u}_2, \ldots, \bar{u}_{n-1}]][\bar{u}^\pm] \otimes \mathbb{Z}[u_{2\sigma}^\pm, a_\sigma]/(2a_\sigma).$$
The classes $\bar{u}_1$, $\ldots$, $\bar{u}_{n-1}$, $\bar{u}^\pm$, and $a_\sigma$ are permanent cycles.  All the differentials in the spectral sequence are determined by the differentials 
\begin{eqnarray*}
d_{2^{k+1} -1} (u_{2\sigma}^{2^{k-1}}) &=&  \bar{u}_k\bar{u}^{2^k-1}a_\sigma^{2^{k+1}-1}, \, \, \, 1 \leq k \leq n-1, \\ 
d_{2^{n+1}-1}(u_{2\sigma}^{2^{n-1}})&=& \bar{u}^{2^n-1}a_\sigma^{2^{n+1}-1}, \, \, \, k = n, \
\end{eqnarray*}
and multiplicative structures. 
\end{thm}

The existence of equivariant orientations renders computations that rely on the slice spectral sequence tractable.  This observation was first made in the solution of the Kervaire invariant problem by Hill, Hopkins, and Ravenel \cite{HHR}.  More recently, Hill, Wang, Xu, and the second author used Theorem~\ref{FullVersionONE} to compute completely the slice spectral sequence of a $C_4$-equivariant height-4 Lubin--Tate theory \cite{HillShiWangXu}.

\subsection{Motivation and Main Results}

Topological $K$-theory is a remarkably useful cohomology theory that has produced important homotopy-theoretic invariants in topology.  Many deep facts in topology have surprisingly simple proofs using topological $K$-theory.  For instance, Adams's original solution \cite{AdamsHopfInvariant} to the Hopf invariant one problem used ordinary cohomology and secondary cohomology operations, but, together with Atiyah \cite{AdamsAtiyahHopfInvariant}, he later discovered a much simpler solution using complex $K$-theory and its Adams operations.  He also studied the real $K$-theory of real projective spaces \cite{AdamsVectorFields} and used it to resolve the vector fields on spheres problem.

In 1966, Atiyah \cite{AtiyahKR} formalized the connection between complex $K$-theory, $KU$, and real $K$-theory, $KO$.  The complex conjugation action on complex vector bundles induces a natural $C_2$-action on $KU$.  Under this action, the $C_2$-fixed points and the homotopy fixed points of $KU$ are both $KO$: 
$$KU^{C_2} \simeq KU^{hC_2} \simeq KO.$$
Furthermore, there is a homotopy fixed point spectral sequence computing the homotopy group of $KO$, starting from the action of $C_2$ on the homotopy group of $KU$: 
$$E_2^{s,t} = H^s(C_2; \pi_t KU) \Longrightarrow \pi_{t-s} KO.$$
The spectrum $KU$, equipped with this $C_2$-action and considered as a $C_2$-spectrum, is called Atiyah's Real $K$-theory $K_\mathbb{R}$.  

The spectrum $KU$ is a complex oriented cohomology theory, which means that there is a map $MU \longrightarrow KU$, where $MU$ is the complex cobordism spectrum.  Early work on $MU$ due to Milnor \cite{Milnor60}, Novikov \cite{Novikov60, Novikov62, Novikov67}, and Quillen \cite{Quillen69} established the complex cobordism spectrum as a critical tool in modern stable homotopy theory, with deep connections to algebraic geometry and number theory through the theory of formal groups \cite{COCTALOS, RavenelGreen, LurieChromatic, PetersonChromatic}.  The complex orientation of $KU$ induces a map of rings 
$$\pi_* MU \longrightarrow \pi_* KU$$
on the level of homotopy groups.  Quillen's \cite{Quillen69} calculation of $\pi_*MU$ shows that the map above produces a one dimensional formal group law over $\pi_*KU$, which turns out to be the multiplicative formal group law $\Gamma_m(x,y) = x + y - xy$.  

Analogously as in the case of $KU$, the complex conjugation action on complex manifolds induces a natural $C_2$-action on $MU$.  This action produces the Real cobordism spectrum $MU_\mathbb{R}$ of Landweber \cite{LandweberMUR}, Fujii \cite{FujiiMUR}, and Araki \cite{Araki}.  The underlying spectrum of $MU_\mathbb{R}$ is $MU$, with the $C_2$-action given by complex conjugation.  

Complex conjugation acts on $KU$ and $MU$ by coherently commutative ($\mathbb{E}_\infty$) maps, making $K_\mathbb{R}$ and $MU_\mathbb{R}$ commutative $C_2$-spectra.  The complex orientation of $KU$ is compatible with the complex conjugation action, and it can be refined to a \textit{Real orientation} 
$$MU_\mathbb{R} \longrightarrow K_\mathbb{R}.$$ 
\vspace{0.01in}

Complex $K$-theory belongs to a more general class of spectra --- the Lubin--Tate spectra --- central to the study of chromatic homotopy theory and the stable homotopy groups of spheres.  These spectra are reverse-engineered from algebra as follows.  Given a formal group law $\Gamma$ of finite height $n$ over a perfect field $k$ of characteristic $p$, Lubin and Tate \cite{LubinTate} showed that $\Gamma$ admits a universal deformation defined over a complete local ring $R$ with residue field $k$.  The ring $R$ is non-canonically isomorphic to $W(k)[[u_1, u_2, \ldots, u_{n-1}]]$, over which the formal group law is characterized by a map 
$$MU_* \longrightarrow W(k)[[u_1, u_2, \ldots, u_{n-1}]][u^\pm].$$  

This map can be shown to be Landweber exact.  Applying the Landweber exact functor theorem yields a complex oriented homology theory represented by a homotopy commutative ring spectrum $E_{(k,\Gamma)}$.  When $k = \mathbb{F}_{p^n}$ and $\Gamma$ is the Honda formal group law, the resulting Lubin--Tate spectrum is commonly called $E_n$, the height $n$ Morava $E$-theory.  Since the height $1$ Morava $E$-theory is $KU^\wedge_p$, the Lubin--Tate spectra can be thought of as the higher height analogues of $K$-theory.  

To endow the Lubin--Tate theories $E_{(k, \Gamma)}$ with coherent multiplicative structures, Goerss, Hopkins, and Miller computed the moduli space of $\mathbb{A}_\infty$- and $\mathbb{E}_\infty$-structures on $E_{(k,\Gamma)}$.  The group $\mathbb{G}_n$ of automorphisms of the formal group law $\Gamma$ naturally acts on $\pi_*E_{(k,\Gamma)}$, and the Goerss--Hopkins--Miller computation demonstrates that there is in fact an action of $\mathbb{G}_n$ on $E_{(k,\Gamma)}$ by $\mathbb{E}_\infty$-ring automorphisms. 

For any closed subgroup $G \subseteq \mathbb{G}_n$, one can use the Goerss--Hopkins--Miller action to construct a homotopy fixed point spectrum $E_{(k,\Gamma)}^{hG}:=F(EG_+, E_{(k,\Gamma)})^{G}$.  There are homotopy fixed point spectral sequences of the form 
$$E_2^{s, t} = H^s(G; \pi_t(E_{(k,\Gamma)})) \Longrightarrow \pi_{t-s}(E_{(k,\Gamma)}^{hG}).$$

The spectra $E_n^{hG}$ turn out to be the essential building blocks of the $p$-local stable homotopy category.  In particular, the homotopy groups $\pi_* E_n^{hG}$ assemble to the stable homotopy groups of spheres.  To be more precise, the chromatic convergence theorem of Hopkins and Ravenel~\cite{RavenelOrangeBook} exhibits the $p$-local sphere spectrum $\mathbb{S}_{(p)}$ as the inverse limit of the chromatic tower 
$$\cdots \longrightarrow L_{E_n} \mathbb{S} \longrightarrow L_{E_{n-1}} \mathbb{S} \longrightarrow \cdots \longrightarrow L_{E_0} \mathbb{S},$$
where each $L_{E_n}\mathbb{S}$ is assembled via the chromatic fracture square
$$\begin{tikzcd}
L_{E_n}\mathbb{S} \ar[d] \ar[r] & L_{K(n)}\mathbb{S} \ar[d] \\ 
L_{E_{n-1}} \mathbb{S} \ar[r] & L_{E_{n-1}}L_{K(n)} \mathbb{S},
\end{tikzcd}$$
where $K(n)$ is the $n$th Morava $K$-theory.

Devinatz and Hopkins~\cite{DevinatzHopkins} proved that $L_{K(n)}\mathbb{S} \simeq E_n^{h\mathbb{G}_n}$, and, furthermore, that the Adams--Novikov spectral sequence computing $L_{K(n)}\mathbb{S}$ can be identified with the associated homotopy fixed point spectral sequence for $E_n^{h\mathbb{G}_n}$.  The fixed point spectra $\{E_n^{hG}\, |\, G \subset \mathbb{G}_n \}$, where $G$ ranges over finite subgroups of $\mathbb{G}_n$, play a central role in resolving $E_n^{h\mathbb{G}_n}$ (see \cite{RavenelGreen, GoerssHennMahowaldRezk, Henn2007, BobkovaGoerss, BeaudryGoerssHenn}).  They are also important in modern detection theorems, which are results about families in the stable homotopy groups of spheres obtained by studying the Hurewicz homomorphism from the sphere spectrum to these periodic theories (\cite{RavenelOdd, HHR}).

For the rest of the paper, we designate $p = 2$.  When $k = \mathbb{F}_2$, and $\Gamma$ is the multiplicative formal group $\Gamma_m(x, y) = x + y - xy$, we find that $E_{(\mathbb{F}_2, \Gamma_m)}^{hC_2} = KO_2^\wedge$, the 2-adic completion of real $K$-theory.  For this reason, the spectra $E_n^{hG}$ are commonly called the \textit{higher real $K$-theories}.  


At height 2, these homotopy fixed points are known as TMF and TMF with level structures.  Computations of the homotopy groups of these spectra are done by Hopkins--Mahowald \cite{HopkinsMahowald}, Bauer \cite{BauerTMF}, Mahowald--Rezk \cite{MahowaldRezk}, Behrens--Ormsby \cite{BehrensOrmsby}, Hill--Hopkins--Ravenel \cite{HHRKH}, and Hill--Meier \cite{HillMeier}.

\vspace{0.2in}
For higher heights $n >2$, the homotopy fixed points $E_n^{hG}$ are notoriously difficult to compute.  Prior to the present work, essentially no progress had been made.  One of the chief reasons that these homotopy fixed points are so difficult to compute is because the group actions are constructed purely from obstruction theory.  This stands in contrast to the cases of Atiyah's Real $K$-theory $K_\mathbb{R}$ and Real cobordism $MU_\mathbb{R}$, whose actions come from geometry.  The main theorem of this work establishes the first known connection between the obstruction-theoretic actions on Lubin--Tate theories and the geometry of complex conjugation:

\begin{thm} \label{FullVersion}
Let $k$ be a perfect field of characteristic $2$, $\Gamma$ a height $n$ formal group law over $k$, and $E_{(k, \Gamma)}$ the corresponding Lubin--Tate theory.  Suppose $G$ is a finite subgroup of the Morava stabilizer group that contains the central subgroup $C_2$.  Then there is a $G$-equivariant map
$$N^{G}_{C_2} MU_\mathbb{R} \longrightarrow E_{(k, \Gamma)},$$
where $N_{C_2}^G(-)$ is the Hill--Hopkins--Ravenel norm functor. 
\end{thm}

In particular, when $G = C_2$, Theorem~\ref{FullVersion} implies that for all height $n \geq 1$, the classical complex orientation $MU \to E_n$ can be refined to a Real orientation 
$$MU_\mathbb{R} \longrightarrow E_n.$$

The presence of geometry, aside from its intrinsic interest, has tremendous computational consequences.  Hu and Kriz \cite{HuKriz} were able to completely compute the homotopy fixed point spectral sequence for $MU_\mathbb{R}$.  Combining our main theorem with the Hu--Kriz computation, we obtain the first calculations for $E_n^{hC_2}$, valid for arbitrarily large heights $n$. 

\begin{thm} \label{FullVersion2}
The $E_2$-page of the $RO(C_2)$-graded homotopy fixed point spectral sequence of $E_n$ is 
$$E_2^{s, t} (E_n^{hC_2}) = W(\mathbb{F}_{2^n})[[\bar{u}_1, \bar{u}_2, \ldots, \bar{u}_{n-1}]][\bar{u}^\pm] \otimes \mathbb{Z}[u_{2\sigma}^\pm, a_\sigma]/(2a_\sigma).$$
The classes $\bar{u}_1$, $\ldots$, $\bar{u}_{n-1}$, $\bar{u}^\pm$, and $a_\sigma$ are permanent cycles.  All the differentials in the spectral sequence are determined by the differentials 
\begin{eqnarray*}
d_{2^{k+1} -1} (u_{2\sigma}^{2^{k-1}}) &=&  \bar{u}_k\bar{u}^{2^k-1}a_\sigma^{2^{k+1}-1}, \, \, \, 1 \leq k \leq n-1, \\ 
d_{2^{n+1}-1}(u_{2\sigma}^{2^{n-1}})&=& \bar{u}^{2^n-1}a_\sigma^{2^{n+1}-1}, \, \, \, k = n, \
\end{eqnarray*}
and multiplicative structures. 
\end{thm}

The existence of equivariant orientations renders computations that rely on the slice spectral sequence tractable.  This observation was first made in the solution of the Kervaire invariant problem by Hill, Hopkins, and Ravenel in 2009.  

In their landmark paper \cite{HHR}, Hill, Hopkins, and Ravenel established that the Kervaire invariant elements $\theta_j$ do not exist for $j \geq 7$ (see also \cite{HaynesKervaire, HHRCDM1, HHRCDM2} for surveys on the result).  A key construction in their proof is the spectrum $\Omega$, which detects the Kervaire invariant elements in the sense that if $\theta_j \in \pi_{2^{j+1}-2} \mathbb{S}$ is an element of Kervaire invariant 1, then the Hurewicz image of $\theta_j$ under the map $\pi_*\mathbb{S} \to \pi_*\Omega$ is nonzero.  

The detecting spectrum $\Omega$ is constructed using equivariant homotopy theory as the $C_8$-fixed point of a genuine $C_8$-spectrum $\Omega_\mathbb{O}$, which in turn is an equivariant localization of $MU^{((C_8))} := N_{C_2}^{C_8} MU_\mathbb{R}$.  In particular, there is a $C_8$-equivariant orientation 
$$MU^{((C_8))} \longrightarrow \Omega_\mathbb{O}.$$

For $G = C_{2^n}$, the $G$-spectrum $MU^{((G))}$ and its equivariant localizations are amenable to computations.  To analyze the $C_8$-equivariant homotopy groups of $\Omega_\mathbb{O}$, Hill, Hopkins, and Ravenel generalized the $C_2$-equivariant filtration of Hu--Kriz \cite{HuKriz} and Dugger \cite{DuggerKR} to a $G$-equivariant Postnikov filtration for all finite groups $G$.  They called this the \textit{slice filtration}.  Given any $G$-equivariant spectrum $X$, the slice filtration produces the slice tower $\{P^*X\}$, whose associated slice spectral sequence strongly converges to the $RO(G)$-graded homotopy groups $\pi_\bigstar^G X$.  

Using the slice spectral sequence, Hill, Hopkins, and Ravenel proved the Gap Theorem and the Periodicity Theorem, which state, respectively, that $\pi_i^{C_8} \Omega_\mathbb{O} = 0$ for $-4 < i < 0$, and that there is an isomorphism $\pi_*^{C_8} \Omega_\mathbb{O} \cong \pi_{*+256}^{C_8} \Omega_\mathbb{O}.$
The two theorems together imply that
$$\displaystyle \pi_{2^{j+1}-2} \Omega = \pi_{2^{j+1}-2}^{C_8} \Omega_\mathbb{O} = 0$$
for all $j \geq 7$, from which the nonexistence of the corresponding Kervaire invariant elements follows.  

Analogues of the Kervaire invariant elements exist at odd primes.  In 1978, Ravenel \cite{RavenelOdd} computed the $C_p$-homotopy fixed points of the Lubin--Tate spectrum $E_{p-1}$ and proved that the $p$-primary Kervaire invariant elements do not exist for all $p \geq 5$.

In light of Ravenel's work, Hill, Hopkins, and Ravenel had hoped that the homotopy fixed points of a certain Lubin--Tate theory would entail the nonexistence of the bona fide Kervaire invariant elements.  Indeed, they mentioned in \cite{HHRCDM2} that the Detection Theorem held for $E_4^{hC_8}$, which made it a promising candidate to resolve the Kervaire invariant problem.  However, because of the computational difficulties surrounding the homotopy fixed point spectral sequence, they could not prove the Gap Theorem and the Periodicity Theorem for $E_4^{hC_8}$. 

Instead, in \cite{HHR}, they opted to consider $\Omega_\mathbb{O}^{C_8}$, which serves to mimic $E_4^{hC_8}$, but benefits from the geometric rigidity it inherits from $MU^{((C_8))}$: once the theory of slice filtrations is set up, the Gap Theorem and the Periodicity Theorem are immediate.  

Despite the computational access granted via $MU^{((G))}$, its localizations are unsuitable for chromatic homotopy theory because the $E_2$-pages of their slice spectral sequences are too large and contain many unnecessary classes.  Thus, one cannot hope to resolve the $K(n)$-local sphere by fixed points of the localizations of $MU^{((G))}$.  

To address this situation, one could hope to quotient $MU^{((G))}$ by generators in its equivariant stable homotopy group in order to cut down the size of its slice spectral sequence and its coefficient group.  This can been done at heights $\leq 2$.  At higher heights, however, the quotienting process fails to preserve the higher coherent structure ($\mathbb{E}_\infty$-ness) of the spectrum.  

For example, even at $G = C_2$, the spectra $BP_\mathbb{R}\langle n \rangle$ and the Real Johnson--Wilson theories $E\mathbb{R}(n)$ are not known to be rings when $n \geq 3$ (see \cite{PhantomRing} and \cite[Remark 4.19]{HillMeier}).  They also have no clear connection to the Lubin--Tate spectra $E_n$.  Therefore, despite its computability, it is difficult to use $E\mathbb{R}(n)$ to obtain information about the higher real $K$-theories $E_n^{hG}$ and the $K(n)$-local sphere.  

Theorem~\ref{FullVersion} and Theorem~\ref{FullVersion2} combine the computational power of the slice spectral sequence with the import of the Lubin--Tate spectra.  Preponderant in chromatic homotopy theory, the Lubin--Tate spectra have smaller coefficient rings than the localizations of $MU^{((G))}$, so they are ideal candidates for resolving the $K(n)$-local sphere.  

It is a consequence of Theorem~\ref{FullVersion2} that $E_n$ is an even $C_2$-spectrum, and, in particular, has pure and isotropic slice cells.  In a forthcoming paper by the second author, Theorem~\ref{FullVersion} and Theorem~\ref{FullVersion2} will be used to compute the slice filtration of $E_n$ for all $n$, considered as a $G$-spectrum, where $G$ is a cyclic group of order a power of 2.  It will follow that $E_n$ has pure and isotropic $G$-slice cells.  

\begin{rmk}\rm
Once this is established, the proofs in \cite{HHR} are applicable to $E_n^{hG}$.  Hence $E_n^{hG}$ satisfies a Gap Theorem and a Periodicity Theorem, and, moreover, there is a factorization
$$\begin{tikzcd}
MU^{((G))} \ar[r] \ar[d] & E_n \\
D^{-1} MU^{((G))} \ar[ru] &.
\end{tikzcd}$$
In particular, there is a $C_8$-equivariant map from the detection spectrum $\Omega_\mathbb{O} \longrightarrow E_4$. 
\end{rmk}

Recently, Hill, Wang, Xu, and the second author used Theorem~\ref{FullVersion} to compute completely the slice spectral sequence of a $C_4$-equivariant height-4 Lubin--Tate theory \cite{HillShiWangXu}.  It is a current project to harness Theorem~\ref{FullVersion} to compute completely the $C_{2^m}$-fixed points of $E_n$ and study its Hurewicz images.

\subsection{Summary of Contents}
We now turn to a more detailed summary of the contents of this paper.  To prove Theorem~\ref{FullVersion}, we first consider a specific Lubin--Tate spectra.  Let $\widehat{E(n)}$ denote the $2$-periodic completed Johnson--Wilson theory, with
$$\pi_*(\widehat{E(n)})=\mathbb{Z}_2[[v_1,v_2,...,v_{n-1}]][u^{\pm}],\text{ } |u|=2.$$
This spectrum is a version of Morava $E$-theory.  In particular, it is a complex-oriented $\mathbb{E}_\infty$-ring spectrum.  Work of Goerss, Hopkins, and Miller \cite{GoerssHopkins, HopkinsMiller} identifies the space of $\mathbb{E}_\infty$-ring automorphisms of $\widehat{E(n)}$, and in particular ensures the existence of a central Galois $C_2$-action by $\mathbb{E}_\infty$-ring maps.  At the level of homotopy groups, $C_2$ acts as the formal inverse of the canonical formal group law.

There is also a natural $C_2$-action on $MU$, by complex-conjugation.  To this end, we first prove the following: 

\begin{thm} \label{MainTheorem}
The spectrum $\widehat{E(n)}$, with its central Galois $C_2$-action, is \textbf{Real oriented}.  That is to say, it receives a $C_2$-equivariant map 
$$MU_\mathbb{R} \longrightarrow \widehat{E(n)}$$
from the Real cobordism spectrum $MU_\mathbb{R}$.
\end{thm}

Leveraging the Hill--Hopkins--Ravenel norm functor \cite{HHR}, Theorem~\ref{FullVersion} is a formal consequence of Theorem \ref{MainTheorem}.

To prove Theorem \ref{MainTheorem} it will be helpful to sketch a construction of $\widehat{E(n)}$ as a ring spectrum, not yet worrying about any $C_2$-actions.  Recall that there is a periodic version of complex cobordism, denoted $MUP$, that is an $\mathbb{E}_\infty$-ring spectrum.  We denote the symmetric monoidal ($\infty$-)category of $MUP$-module spectra by $\text{MUP\textbf{-Mod}}$.  The subgroupoid spanned by the unit and its automorphisms is the space $BGL_1(MUP)$, which is naturally an infinite loop space.  Associated to any map of spaces $f:X \rightarrow BGL_1(MUP)$ is a Thom $MUP$-module $\text{Thom}(f)$ \cite{ThomInfinity}.  The category of spaces over $BGL_1(MUP)$ is symmetric monoidal, and an associative algebra object in this category gives rise to an $\mathbb{A}_\infty$-algebra structure on its Thom spectrum \cite{ThomRing}.

Consider now the following diagram of categories:

\begin{equation*}
\begin{tikzcd} [column sep=large] \tag{$\star$}
\mathbb{A}_\infty\left(\textbf{Spaces}_{/BGL_1(MUP)}\right) \arrow{r}{Thom}& \mathbb{A}_\infty(\text{MUP}\textbf{-Mod}) \arrow{d}{L_{K(n)}} \arrow{r}{Forget} & \mathbb{A}_\infty(\textbf{Spectra}) \arrow{d}{L_{K(n)}}  \\
\textbf{Spaces}_{/B^2GL_1(MUP)} \arrow{u}{\Omega} & \mathbb{A}_\infty(\text{MUP}\textbf{-Mod}) \arrow{r}{Forget} & \mathbb{A}_\infty(\textbf{Spectra}) \\
 && \mathbb{E}_\infty(\textbf{Spectra}). \arrow[swap]{u}{Forget}
\end{tikzcd}
\end{equation*}

In Section \ref{ClassicalConstruction}, we will construct a certain map of spaces $X \rightarrow B^2GL_1(MUP)$.  Applying $\Omega$ and then the Thom spectrum construction, we obtain an $\mathbb{A}_\infty$-MUP-algebra $E(n)$ that is a $2$-periodic version of Johnson--Wilson theory.  The $K(n)$-localization of $E(n)$ is $\widehat{E(n)}$, equipped with the structure of an $\mathbb{A}_\infty$-MUP-algebra.

It is a consequence of work of Goerss, Hopkins, and Miller \cite{GoerssHopkins, HopkinsMiller} that we may lift the $\mathbb{A}_\infty$-ring spectrum underlying $\widehat{E(n)}$ to an $\mathbb{E}_\infty$-ring spectrum.  Indeed, letting $\mathcal{C}^{\simeq}$ denote the maximal subgroupoid of a category $\mathcal{C}$, they prove that the path-component of $\mathbb{A}_\infty(\textbf{Spectra})^{\simeq}$ containing $\widehat{E(n)}$ is \textit{equivalent} to a path component in $\mathbb{E}_\infty(\textbf{Spectra})^{\simeq}$, with the equivalence given by the forgetful functor.

Our strategy for the proof of Theorem \ref{MainTheorem} is to produce a Real orientation $MU_\mathbb{R} \rightarrow \overline{E(n)}$ into \textit{some} ring spectrum $\overline{E(n)}$ with $C_2$-action.  The $\overline{E(n)}$ we produce is obviously equivalent to $\widehat{E(n)}$ as a spectrum, and the $C_2$-action is obviously the Galois one up to homotopy.  However, it is \textit{not} at all obvious that the full, coherent $C_2$-action on $\overline{E(n)}$ is the Galois action.  To prove it, we must make full use of the Goerss--Hopkins--Miller theorem.

We produce $\overline{E(n)}$ via a $C_2$-equivariant lift of the above construction of $\widehat{E(n)}$:

\begin{cnstr} \label{MainConstruction}
In section \ref{ConstructionSection}, each of the categories in the diagram $(\star)$ will be equipped with a $C_2$-action, yielding an \textit{equivariant} diagram:

\begin{equation*} \label{equidiagram}
\begin{tikzcd} [column sep=large] \tag{$\star \star$}
\mathbb{A}_\infty\left(\textbf{Spaces}_{/BGL_1(MUP)}\right) \arrow[loop,distance=2em] \arrow{r}{Thom}& \mathbb{A}_\infty(MUP\textbf{-Mod}) \arrow[loop,distance=2em] \arrow{d}{L_{K(n)}} \arrow{r}{Forget} & \mathbb{A}_\infty(\textbf{Spectra}) \arrow[loop,distance=2em,swap]{}{op} \arrow{d}{L_{K(n)}}\\
\textbf{Spaces}_{/B^\rho GL_1(MUP)} \arrow[out=-130, in=-50, loop, distance=2em] \arrow{u}{\Omega^{\sigma}} & \mathbb{A}_\infty(MUP\textbf{-Mod}) \arrow[out=-130, in=-50, loop, distance=2em] \arrow{r}{Forget} & \mathbb{A}_\infty(\textbf{Spectra}) \arrow[out=-12, in=12, loop, distance=1.5em, swap]{}{\text{op}}  \\
&& \mathbb{E}_\infty(\textbf{Spectra}). \arrow[out=-130, in=-50, loop, distance=2em, swap]{}{trivial} \arrow[swap]{u}{Forget}
\end{tikzcd}
\end{equation*}

The action on $\mathbb{E}_\infty(\textbf{Spectra})$ will be the trivial $C_2$-action.  The action on $\mathbb{A}_\infty(\textbf{Spectra})$ will be the non-trivial $\text{op}$ action that takes an algebra to its opposite.
\end{cnstr}

\begin{rmk} By a homotopy fixed point in a category $\mathcal{C}$ with $C_2$-action we mean an object in the category $\mathcal{C}^{hC_2}$.  For example, a homotopy fixed point in $\mathbb{E}_\infty(\textbf{Spectra})$ with its trivial action is just an $\mathbb{E}_\infty$-ring spectrum with $C_2$-action by $\mathbb{E}_\infty$-ring maps.  A homotopy fixed point for the op action on $\mathbb{A}_\infty(\textbf{Spectra})$ is an $\mathbb{A}_\infty$-algebra $A$ equipped with an \textit{involution}, meaning a coherent algebra map $\sigma:A \rightarrow A^{\text{op}}$.  This is a Borel equivariant version of what other authors have called an $\mathbb{E}_\sigma$-ring, or a ring with anti-involution (see \cite{HillEquivDisc, DottoRealTHH}).  We believe the use of algebras with involutions to be the most interesting feature of our construction.
\end{rmk}

In Sections \ref{ThomSection} and \ref{sec:Thm1.1}, we will refine our map $X \rightarrow B^2GL_1(MUP)$ to an equivariant map $X \rightarrow B^{\rho}GL_1(MU_\mathbb{R}P).$  Applying $\Omega^{\sigma}$ produces a homotopy fixed point of $\mathbb{A}_\infty\left(\textbf{Spaces}_{/BGL_1(MUP)}\right)$, which in turn equips $E(n)$ with an $\mathbb{A}_\infty$-involution. After $K(n)$-localizing, we obtain a $C_2$-action on $\widehat{E(n)}$ by $\mathbb{A}_\infty$-involutions.  The Goerss--Hopkins--Miller Theorem \cite{GoerssHopkins, HopkinsMiller} proves that any such action on $\widehat{E(n)}$ may be lifted to one by $\mathbb{E}_\infty$-ring maps.  Since Goerss, Hopkins, and Miller furthermore calculate the entire space of $\mathbb{E}_\infty$-ring automorphisms of $\widehat{E(n)}$, we may determine any $\mathbb{E}_\infty$-$C_2$-action on $\widehat{E(n)}$ by its effect on homotopy groups.\\

In Section~\ref{sec:HFPSSEn}, we look towards computational applications of the above results.  For simplicity, we use a specific Morava $E$-theory $E_n$ that is defined via a lift of the height $n$ Honda formal group law over $\mathbb{F}_{2^n}$.  Its homotopy groups are 
$$\pi_*E_n = W(\mathbb{F}_{2^n})[[u_1, u_2, \ldots, u_{n-1}]][u^\pm].$$
Using Theorem~\ref{FullVersion} and leveraging Hu and Kriz's computation of the homotopy fixed point spectral sequence for $MU_\mathbb{R}$ \cite{HuKriz}, we prove Theorem~\ref{FullVersion2}.  As a corollary, we learn that as a $C_2$-spectrum, $E_n$ is \textit{strongly even} and \textit{Real Landweber exact} in the sense of Hill--Meier \cite{HillMeier}. 

\begin{thm}
$E_n$ is strongly even and Real Landweber exact.  More precisely, $\underline{\pi}_{k\rho-1}E_n = 0$ and $\underline{\pi}_{k\rho}E_n$ is a constant Mackey functor for all $k \in \mathbb{Z}$.  The Real orientation $MU_\mathbb{R} \to E_n$ induces a map 
$${MU_\mathbb{R}}_\bigstar(X) \otimes_{MU_{2*}} ({E_n})_{2*} \to {E_n}_\bigstar(X)  $$
which is an isomorphism for every $C_2$-spectrum $X$.  
\end{thm}

The second author's detection theorem for $MU_\mathbb{R}^{hC_2}$, joint with Li, Wang, and Xu \cite{HurewiczImages}, allows us to conclude a detection theorem for $E_n^{hC_2}$.  Roughly speaking, as the height grows, an increasing amount of the Kervaire and $\bar{\kappa}$-families in the stable homotopy groups of spheres are detected by $\pi_* E_n^{hC_2}$.  More precisely, we prove in Section~\ref{sec:HurewiczImages} the following:

\begin{thm} [Detection theorem for $E_n^{hC_2}$] \hfill
\begin{enumerate}
\item For $1 \leq i, j \leq n$, if the element $h_{i} \in \Ext_{\mathcal{A}_*}^{1, 2^i}(\F, \F)$ or $h_{j}^2 \in \Ext_{\mathcal{A}_*}^{2, 2^{j+1}}(\F, \F)$ survives to the $E_\infty$-page of the Adams spectral sequence, then its image under the Hurewicz map $\pi_*\mathbb{S} \to \pi_* E_n^{hC_2}$ is nonzero.  
\item For $1 \leq k \leq n-1$, if the element $g_k \in \Ext_{\mathcal{A}_*}^{4, 2^{k+2}+2^{k+3}}(\F, \F)$ survives to the $E_\infty$-page of the Adams spectral sequence, then its image under the Hurewicz map $\pi_*\mathbb{S} \to \pi_*E_n^{hC_2}$ is nonzero. 
\end{enumerate}
\end{thm}

\begin{rmk}
We freely use the language of $\infty$-categories throughout this work, and will refer to an $\infty$-category simply as a category.  If $\mathcal{C}$ is a symmetric monoidal category, we use $\mathbb{A}_\infty(\mathcal{C})$ to denote the category of associative algebra objects in $\mathcal{C}$, and similarly use $\mathbb{E}_\infty(\mathcal{C})$ to denote commutative algebra objects.  We will use \textbf{Spaces} to denote the symmetric monoidal category of \textit{pointed} spaces under cartesian product.
\end{rmk}

\subsection{Acknowledgements} 
The authors would like to thank Jun-Hou Fung, Nitu Kitchloo, Guchuan Li, Vitaly Lorman, Lennart Meier, Denis Nardin, Eric Peterson, Doug Ravenel, David B Rush, Jay Shah, Guozhen Wang, Zhouli Xu, and Allen Yuan for helpful conversations.  We also thank Hood Chatham for his comprehensive and easy to use spectral sequence package, which produced all of our diagrams.  Lastly, and most importantly, we owe tremendous debts to Mike Hopkins and Mike Hill, both of whom offered crucial guidance at various stages of the project.  The first author's work was supported by an NSF GRFP fellowship under Grant DGE-$1144152$.

\section{Thom Spectra and Johnson--Wilson Theory} \label{ClassicalConstruction}

In this section we will describe a non-equivariant construction of $\widehat{E(n)}$, a Landweber exact Morava $E$-theory with
$$\pi_*(\widehat{E(n)}) \cong \mathbb{Z}_2[[v_1,v_2,\cdots,v_{n-1}]][u^{\pm}].$$
Our construction is a riff on Theorem $1.4$ of \cite{ThomQuotient}.

We begin with a brief review of the classical theory of Thom spectra.  Useful references, in the language of $\infty$-categories we espouse here, include \cite{ThomInfinity} and \cite{ThomRing}.

If $R$ is an $\mathbb{E}_\infty$-ring spectrum, then the category of $R$-modules acquires a symmetric monoidal structure.  The subcategory consisting of the unit and its automorphisms is denoted $BGL_1(R)$.  The symmetric monoidal structure equips $BGL_1(R)$ with an infinite loop space structure, and we write $BGL_1(R) \simeq \Omega^{\infty} \Sigma gl_1(R)$.  The space $GL_1(R) \simeq \Omega^{\infty} gl_1(R)$ sits in a pullback square
$$
\begin{tikzcd}
GL_1(R) \arrow{r} \arrow{d} & \Omega^{\infty} R \arrow{d} \\
\pi_0(R)^{\times} \arrow{r} & \pi_0(R),
\end{tikzcd}
$$
where $\pi_0(R)^{\times}$ is the subset of units of $\pi_0(R)$ under multiplication.  From this latter description of $GL_1(R)$, it is clear that 
$$\pi_*(BGL_1(R)) \cong \pi_{*-1}(GL_1(R)) \cong \pi_{*-1}(R), \text{ for } *>1.$$

Given a map of spaces $X \rightarrow BGL_1(R)$, we can form the Thom $R$-module by taking the colimit of the composite functor $X \rightarrow BGL_1(R) \subset R\textbf{-Mod}$.  If $X$ is a loop space and $X \rightarrow BGL_1(R)$ is a loop map, then the main theorem of \cite{ThomRing} shows that the associated Thom spectrum is an $\mathbb{A}_\infty$-$R$-algebra.  Similarly, if $X$ is an infinite loop space and $X \rightarrow BGL_1(R)$ an infinite loop map, then \cite{ThomRing} shows that the associated Thom spectrum is an $\mathbb{E}_\infty$-$R$-algebra.

Given two maps $f_1:X_1 \rightarrow BGL_1(R)$ and $f_2:X_2 \rightarrow BGL_1(R)$, we may use the infinite loop space structure on $BGL_1(R)$ to produce a product map
$$(f_1,f_2):X_1 \times X_2 \rightarrow BGL_1(R) \times BGL_1(R) \rightarrow BGL_1(R).$$
The Thom $R$-module $\text{Thom}(f_1,f_2)$ is the $R$-module smash product $\text{Thom}(f_1) \smsh_R \text{Thom}(f_2)$.

We may speak not only of $BGL_1(R)$, but also of the infinite loop space $\text{Pic}(R)$.  As a symmetric monoidal category, $\text{Pic}(R)$ is the full subcategory of $R\textbf{-Mod}^{\simeq}$ spanned by the invertible $R$-modules.  It is a union of path components each of which is equivalent to $BGL_1(R)$.  Again, \cite{ThomRing} explains that the colimit of an infinite loop map $X \rightarrow \text{Pic(R)} \subset R\textbf{-Mod}$ is an $\mathbb{E}_\infty$-$R$-algebra.  Our only use of this more general construction is to recall the following classical example:

\begin{exm} \label{MUPConstruction}
The complex $J$-homomorphism is an infinite loop map $BU \times \mathbb{Z} \rightarrow \text{Pic}(\mathbb{S})$, obtained via the algebraic $K$-theory construction on $\coprod BU(n) \rightarrow \text{Pic}(\mathbb{S})$.  The resulting Thom $\mathbb{E}_\infty$-ring spectrum is the periodic complex cobordism spectrum, denoted $MUP$.  The $2$-connective cover of spectra $bu \rightarrow ku$ yields an infinite loop map $BU \rightarrow BU \times \mathbb{Z}$, which induces a map of Thom $\mathbb{E}_\infty$-ring spectra $MU \rightarrow MUP$.

The map $J:BU \times \mathbb{Z} \rightarrow \text{Pic}(\mathbb{S})$ decomposes as a product of the infinite loop map $BU \rightarrow \text{BGL}_1(\mathbb{S})$ and the loop map $\mathbb{Z} \rightarrow \text{Pic}(\mathbb{S})$.  This yields an equivalence of Thom $\mathbb{A}_\infty$-ring spectra
$$MUP \simeq  MU \smsh \left(\bigvee_{n \in \mathbb{Z}} S^{2n} \right) \simeq \bigvee_{n \in \mathbb{Z}} \Sigma^{2n} MU,$$
which allows us to calculate $\pi_*(MUP) \cong \pi_*(MU)[u^{\pm}] \cong \mathbb{Z}[x_1,x_2,\cdots][u^{\pm}],$ where $|u|=2$ and $|x_i|=2i$.
The complex-conjugation action on $BU \times \mathbb{Z}$ by infinite loop maps yields a $C_2$-action on $MUP$ by $\mathbb{E}_\infty$-ring homomorphisms; we will make no use of this action in the current section, but much use of it in Sections \ref{ConstructionSection} and \ref{ThomSection}.
\end{exm}

We now specialize the discussion and embark on our construction of $E(n)$.  Suppose that we choose a non-zero $\alpha \in \pi_2(MUP) \cong \pi_3(BGL_1(MUP))$.  Then, e.g. by \cite[Theorem 5.6]{ThomQuotient} or \cite[Theorem 4.10]{ThomRing}, there is an equivalence of $MUP$-module spectra
$$\text{Thom}(\alpha) \simeq \text{Cofiber}(\Sigma^2 MUP \stackrel{\alpha}{\rightarrow} MUP) \simeq MUP/\alpha.$$

If we choose a sequence of elements $(\alpha_1,\alpha_2, \cdots, \alpha_n) \in \pi_2(MUP),$ we may produce a map $$S^3 \times S^3 \times \cdots S^3 \rightarrow BGL_1(MUP)$$ and an associated Thom $MUP$-module
\begin{eqnarray*}
\text{Thom}(\alpha_1,\alpha_2, \cdots, \alpha_n) &\simeq& (MUP/\alpha_1) \smsh_{MUP} (MUP/\alpha_2) \smsh_{MUP} \cdots \smsh_{MUP} (MUP/\alpha_n) \\
&\simeq& MUP/(\alpha_1,\alpha_2,\cdots,\alpha_n).
\end{eqnarray*}

If the sequence $(\alpha_1,\alpha_2,\cdots, \alpha_n)$ is \textit{regular} in $\pi_*(MUP)$, then the usual cofiber sequences imply that
$$\pi_*(MUP/(\alpha_1,\alpha_2,\cdots,\alpha_n)) \cong \pi_*(MUP)/(\alpha_1,\alpha_2,\cdots,\alpha_n).$$
Finally, we may even mod out an infinite regular sequence $(\alpha_1,\alpha_2, \cdots)$ by using the natural maps $$S^3 \rightarrow S^3 \times S^3 \rightarrow S^3 \times S^3 \times S^3 \rightarrow \cdots$$ to produce a filtered colimit of $MUP$-modules
$$MUP/\alpha_1 \rightarrow MUP/(\alpha_1,\alpha_2) \rightarrow MUP/(\alpha_1,\alpha_2,\alpha_3) \rightarrow \cdots \rightarrow MUP/(\alpha_1,\alpha_2, \cdots).$$

\begin{prop} \label{ClassicLoop}
Each map $\alpha_i:S^3 \rightarrow BGL_1(MUP)$ can be given the structure of a loop map.  In other words, the above construction of the $MUP$-module $MUP/(\alpha_1,\alpha_2,\cdots)$ can be refined to a construction of an $\mathbb{A}_\infty$-$MUP$-algebra.
\end{prop}

\begin{proof}
It will suffice to construct a map $\widetilde{\alpha_i}:BS^3 \rightarrow B^2GL_1(MUP)$ such that $\Omega \widetilde{\alpha_i} \simeq \alpha_i$.  This is equivalent to asking that the precomposition of the map $\widetilde{\alpha_i}:BS^3 \rightarrow B^2GL_1(MUP)$ with the inclusion $S^4 \rightarrow BS^3$ be adjoint to the map $\alpha_i:S^3 \rightarrow BGL_1(MUP)$.  In fact, any map $S^4 \rightarrow B^2GL_1(MUP)$ automatically admits at least one factorization through $BS^3$.  The reason is that $BS^3$ admits an even cell decomposition: there is a filtered colimit
$$S^4=Y_1 \rightarrow Y_2 \rightarrow Y_3 \rightarrow \cdots \rightarrow BS^3$$
and pushouts
$$
\begin{tikzcd}
S^{4n-1} \arrow{d} \arrow{r} & Y_{n-1} \arrow{d} \\
D^{4n} \arrow{r} & Y_{n}.
\end{tikzcd}
$$
This cell decomposition is easily seen from the model $BS^3 \simeq \mathbb{HP}^\infty$, the infinite dimensional quaternionic projective space, where it is the canonical cell-decomposition corresponding to the inclusions of the $\mathbb{HP}^{\ell}$.
The obstructions to factoring a map $Y_{n-1} \rightarrow B^2GL_1(MUP)$ through $Y_n$ therefore live in $\pi_{4n-1}B^2GL_1(MUP) \cong \pi_{4n-3} (MUP)$.  This group is $0$, as explained in Example \ref{MUPConstruction}.
\end{proof}

To summarize, if we choose any regular sequence $(\alpha_1,\alpha_2,\cdots) \in \pi_*(MUP) \cong \mathbb{Z}[x_1,x_2,\cdots][u^{\pm}]$, each element of which lies in degree $2$, then we may construct the quotient $MUP$-module $MUP/(\alpha_1,\alpha_2,\cdots)$ as an $\mathbb{A}_\infty$-$MUP$-algebra.  This fact was originally proven without the use of Thom spectra by Angeltveit \cite[Corollary~3.7]{AngeltveitAinfty}, who instead used an obstruction theory developed by Robinson \cite{RobinsonAinfty}.  The following standard lemma allows us to use Proposition \ref{ClassicLoop} to build Morava $E$-theories as $\mathbb{A}_\infty$-algebras:

\begin{lem} \label{fglLemma}
Let $\Gamma$ denote a formal group law of height $n$ over the field $\mathbb{F}_2$, and $E$ the associated Morava $E$-theory.  Then there is a map $MUP \rightarrow E$, classifying a universal deformation of $\Gamma$, which may be described as first taking the quotient of $MUP$ by a regular sequence $(\alpha_1,\alpha_2,\cdots)$ of degree $2$ classes and then performing $K(n)$-localization.
\end{lem}

\begin{rmk}
If the reader prefers, they will lose no intuition by thinking of the regular sequence $$(\alpha_1,\alpha_2,\cdots) = (x_{2^n-1}u^{2-2^n}-u, x_2u^{-1},x_4u^{-3},x_5u^{-4},x_6u^{-5},x_8u^{-7}, \cdots),$$ where the classes $x_iu^{-i+1}$ that are included are those such that either
\begin{itemize}
\item $i$ is not one less than a power of $2$, \text{or}
\item $i$ is larger that $2^{n}-1$.
\end{itemize}
However, since there are non-isomorphic formal groups over $\mathbb{F}_2$, not every Morava $E$-theory is obtained by quotienting out this particular sequence.
\end{rmk}

\begin{dfn}
We denote by $E(n)$ the quotient of $MUP$ by the regular sequence $(\alpha_1,\alpha_2,\cdots)$ of Lemma \ref{fglLemma}, and say that $E(n)$ is a $2$-periodic form of Johnson--Wilson theory.  Proposition \ref{ClassicLoop} provides a (not necessarily unique) construction of $E(n)$ as an $\mathbb{A}_\infty$ $MUP$-algebra.  We are deliberately vague about which formal group law $\Gamma$ defines $E(n)$, so that we may handle all cases at once.
\end{dfn}

\begin{proof}[Proof of Lemma \ref{fglLemma}]
The formal group law $\Gamma$ is classified by some map of (ungraded) rings $\pi_*(BP) \rightarrow \mathbb{F}_2$.  View this map as the solid arrow in the diagram of ring homomorphisms
$$
\begin{tikzcd}
\pi_*(BP) \arrow[dashed]{rrr}{L_2} \arrow[dashed]{rrrd}{L_1} \arrow{rrrdd} & & & W(k)[[u_1,u_2,\cdots,u_{n-1}]] \arrow{d} \\
& & & \mathbb{F}_2[u_1,u_2,\cdots,u_{n-1}]/\mathfrak{m}^2 \arrow{d} \\
& & & \mathbb{F}_2.
\end{tikzcd}
$$
According to \cite[\S 5.10]{HopkinsMiller}, as long as the lift $L_1$ is chosen correctly, any further lift $L_2$ will classify a universal deformation.  Furthermore, we may assume that $v_i$ maps to $u_i$ for $i \le n-1$ under $L_1$, while $v_n$ maps to $1$.  Each $v_j$ for $j>n$ then maps to some $L_1(v_j)$ that is an $\mathbb{F}_2$-linear combination of $L_1(v_1),L_1(v_2),\cdots L_1(v_n)$.  Write $\phi(v_j)$ to denote the element of $\pi_*(BP)$ that is given by the same linear combination of $v_1,v_2,\cdots v_n$.  Then the map $L_2$ can be chosen to be the quotient by the regular sequence $(v_n-1,v_{n+1}-\phi(v_{n+1}),v_{n+2}-\phi(v_{n+2}),\cdots)$.

Using the invertible element $u$ to move elements by even degrees, we may identify $\pi_2(MUP)$ with $\pi_*(MU)$.  Inside of $\pi_*(MU)$ we identify $\pi_*(BP)$ by viewing $v_i$ as $x_{2^i-1}$.  This allows us to talk about $\phi(v_j)$ as a class in $\pi_2(MUP)$.

To obtain the lemma, one mods out by the regular sequence $(\alpha_1,\alpha_2,\cdots) \in \pi_2(MUP)$, where one mods out, in order of $i \in \mathbb{N}$:
\begin{itemize}
\item All $x_iu^{-i+1}$ with $i$ not one less than a power of $2$.
\item The class $x_{i}u^{-i+1}-u$ where $i=2^{n}-1$.
\item The classes $x_iu^{-i+1}-\phi(v_j)$ where $i=2^j-1$ for $j>n$.
\end{itemize}
\end{proof}

\section{Categories with Involutions} \label{ConstructionSection}

Recall that a key fact powering the constructions of the previous sections is that, if $R$ is an $\mathbb{E}_\infty$-ring, then any map of loop spaces
$$\Omega X \longrightarrow \Omega B^2GL_1(R) \simeq BGL_1(R)$$
induces an $\mathbb{E}_1$-Thom spectrum.

In the $C_2$-equivariant setting, there are two sorts of loops one can take.  If $X$ is a $C_2$-space, one can again take ordinary loops, denoted $\Omega X$, but one can also consider based maps from the representation sphere $S^{\sigma}$, which is denoted $\Omega^{\sigma} X$.  If $R$ is a $C_2$-$\mathbb{E}_\infty$-ring spectrum, we shall be interested in $C_2$-equivariant maps
$$\Omega^{\sigma} X \longrightarrow \Omega^{\sigma} B^{\rho} GL_1(R) \simeq BGL_1(R).$$
The structure carried by such a Thom spectrum is no longer that of an $\mathbb{E}_1$-ring, but rather that of a so-called $\mathbb{E}_\sigma$-ring, or \textit{ring with involution}, or to some authors a \textit{ring with anti-involution}.  A number of independent, detailed discussions of $\mathbb{E}_\sigma$-algebras have appeared since the first edition of this preprint was released, including particularly nice exposition in \cite[\S 2.2]{HillEquivDisc} and \cite[\S 2]{DottoRealTHH}.  The discussion of equivariant Thom spectra in \cite{BehrensWilson} is also very relevant.  All of these authors work in the genuine $C_2$-equivariant category, and prove general theorems of which we only need and discuss Borel equivariant shadows.  In this section, we include our original discussion of Borel equivariant $\mathbb{E}_\sigma$-algebras, and in particular note the existence of a diagram of categories with $C_2$-action and equivariant functors between them:

\begin{equation*}
\begin{tikzcd} [column sep=large] \tag{\ref{equidiagram}}
\mathbb{A}_\infty\left(\textbf{Spaces}_{/BGL_1(MUP)}\right) \arrow[loop,distance=2em] \arrow{r}{Thom}& \mathbb{A}_\infty(MUP\textbf{-Mod}) \arrow[loop,distance=2em] \arrow{d}{L_{K(n)}} \arrow{r}{Forget} & \mathbb{A}_\infty(\textbf{Spectra}) \arrow[loop,distance=2em,swap]{}{op} \arrow{d}{L_{K(n)}}\\
\textbf{Spaces}_{/B^\rho GL_1(MUP)} \arrow[out=-130, in=-50, loop, distance=2em] \arrow{u}{\Omega^{\sigma}} & \mathbb{A}_\infty(MUP\textbf{-Mod}) \arrow[out=-130, in=-50, loop, distance=2em] \arrow{r}{Forget} & \mathbb{A}_\infty(\textbf{Spectra}) \arrow[out=-12, in=12, loop, distance=1.5em, swap]{}{\text{op}}  \\
&& \mathbb{E}_\infty(\textbf{Spectra}). \arrow[out=-130, in=-50, loop, distance=2em, swap]{}{trivial} \arrow[swap]{u}{Forget}
\end{tikzcd}
\end{equation*}

\begin{rmk}
For us, an equivariant functor $F:\mathcal{C}_1 \rightarrow \mathcal{C}_2$ between categories with $C_2$-action is an arrow in the functor category $\Hom(BC_2,\textbf{Cat}_\infty)$.  Such an arrow contains a substantial amount of data, and is in particular \textit{not} determined by its underlying functor $F_{\text{underlying}}$ of non-equivariant categories.  For example, using $$\sigma_i:\mathcal{C}_i \rightarrow \mathcal{C}_i$$ to denote the $C_2$-action on $\mathcal{C}_i$, part of the data of $F$ is a choice of natural isomorphism
$$\sigma_2 \circ F_{\text{underlying}} \simeq F_{\text{underlying}} \circ \sigma_1.$$
\end{rmk}

\begin{rmk}
The reader who is comfortable with the existence of the above diagram of categories with (Borel) $C_2$-actions may safely skip the remainder of this section.  The rest of the section only discusses the meaning and existence of the above diagram, which we view as a summary of the formal (i.e., purely categorical) input we need for our main theorem.
\end{rmk}

Let $\textbf{MonCat}_{Lax}$ denote the category of monoidal categories and lax monoidal functors.  In the language of \cite[\S 4.1]{HA}, this is the category of coCartesian fibrations of $\infty$-operads ${\mathcal{C}^{\otimes} \rightarrow \mathcal{A}\text{ssoc}^{\otimes}}$, with morphisms maps of $\infty$-operads $\mathcal{C}_1 \rightarrow \mathcal{C}_2$ over ${\mathcal{A}\text{ssoc}^{\otimes}}$ that are not required to preserve coCartesian arrows. Remark $4.1.1.7$ in \cite{HA} constructs a $C_2$-action $rev$ on $\textbf{MonCat}_{Lax}$.  If $(\mathcal{C},\otimes)$ is a monoidal $\infty$-category, then $(\mathcal{C}_{rev},\otimes_{rev})$ has the same underlying category as $\mathcal{C}$ but the \textit{opposite} $\otimes$-structure, with $X \otimes_{rev} Y$ in $\mathcal{C}_{rev}$ calculated as $Y \otimes X$ in $\mathcal{C}$.  We call a homotopy fixed point for this $\text{rev}$ action a monoidal category $(\mathcal{C},\otimes)$ \textit{with involution}.  Such a category is equipped with a coherent equivalence $\mathcal{C} \stackrel{\simeq}{\rightarrow} \mathcal{C}_{rev}$.  This should be contrasted with a fixed point for the trivial action on $\textbf{MonCat}_{Lax}$, which would just be a monoidal category with $C_2$-action via monoidal functors.

Remark $4.1.1.7$ also constructs an equivalence between $\mathbb{A}_\infty$-algebra objects $A$ in $\mathcal{C}$ and $\mathbb{A}_\infty$-algebra objects $A^{\text{rev}}$ in $\mathcal{C}_{rev}$.  If $\mathcal{C}$ is equipped with an involution, then there is an induced $C_2$-action on $\mathbb{A}_\infty(\mathcal{C})$.  In other words, there is an equivariant functor
$$
\begin{tikzcd} [column sep=large]
\textbf{MonCat}_{Lax} \arrow[loop,distance=2em, swap]{}{rev} \arrow{r}{\mathbb{A}_\infty(\--)} & \textbf{Cat}_{\infty} \arrow[loop,distance=2em, swap]{}{trivial},
\end{tikzcd}
$$
and so a homotopy fixed point in $\textbf{MonCat}_{Lax}$ is sent to one in $\textbf{Cat}_{\infty}$.

Finally, we also consider the category $\textbf{SymMonCat}_{Lax}$ of symmetric monoidal categories and lax functors.  The last paragraph of Remark $4.1.1.7$ of \cite{HA} ensures that the sequence of forgetful functors
$$
\begin{tikzcd} [column sep=large]
\textbf{SymMonCat}_{Lax} \arrow{r}{\text{Forget}} & \textbf{MonCat}_{Lax} \arrow{r}{\text{Forget}} & \textbf{Cat}_{\infty}
\end{tikzcd}
$$
is equivariant, with the trivial $C_2$-action on $\textbf{SymMonCat}_{Lax}$, the rev action on $\textbf{MonCat}_{Lax}$, and the trivial action on $\textbf{Cat}_{\infty}$.

\begin{exm} \label{SymmetricMonoidalExample}
If $\mathcal{C}$ is any symmetric monoidal category, then the trivial action on $\mathcal{C}$ by symmetric monoidal identity functors induces an involution, and therefore an op action on $\mathbb{A}_\infty(\mathcal{C})$.  When we want to emphasize that we are considering $\mathbb{A}_{\infty}(\mathcal{C})$ as a $C_2$-equivariant category with its op action, we will write it as $\mathbb{A}_{\infty}^{\sigma}(\mathcal{C})$. There is an equivariant sequence of categories
$$
\begin{tikzcd}[column sep=large]
\mathbb{E}_\infty(\mathcal{C}) \arrow{r}{Forget} & \mathbb{A}^{\sigma}_\infty(\mathcal{C}) \arrow{r}{Forget} & \mathcal{C},
\end{tikzcd}
$$
where $\mathbb{E}_\infty(\mathcal{C})$ and $\mathcal{C}$ are given the trivial $C_2$-actions and $\mathbb{A}^{\sigma}_\infty(\mathcal{C})$ is given the op action.
\end{exm}

\begin{exm}
Consider the category $\textbf{Set}$ of sets, equipped with the cartesian symmetric monoidal structure.  The trivial $C_2$-action on $\textbf{Set}$ by symmetric monoidal identity functors allows us to view $\textbf{Set}$ as a homotopy fixed point for the trivial $C_2$-action on $\textbf{SymMonCat}_{Lax}$.  This equips the underlying monoidal category of $\textbf{Set}$ with a canonical involution, which in turn equips the category of monoids with a $C_2$-action.  This is the classical $\text{op}$ action that takes a monoid $M$ to its opposite monoid $M^{\text{op}}$, which has the same underlying set but the opposite multiplication.
\end{exm}

Taking $\mathcal{C}$ in Example \ref{SymmetricMonoidalExample} to be the category $\textbf{Spaces}$ of pointed spaces, we obtain the \text{op} action on $\mathbb{A}_\infty(\textbf{Spaces})$, written $\mathbb{A}^{\sigma}_\infty(\textbf{Spaces})$.  We call a homotopy fixed point for this action an $\mathbb{A}_\infty$-space \textit{with involution}; such spaces, considered as groupoids, are special cases of categories with involution.  Any spectrum $E$ with $C_2$-action has an underlying $\mathbb{A}_\infty$-space with involution $\Omega^{\infty} E$.

\begin{exm} \label{OverCat}
Suppose that $X$ is an $\mathbb{A}_\infty$-space with involution.  Then the monoidal category $\textbf{Spaces}_{/X}$ is equipped with an involution.  Concretely, this involution takes an algebra map $A \rightarrow X$ to the natural algebra map $A^{\text{op}} \rightarrow X^{\text{op}} \stackrel{\text{inv}}{\longrightarrow} X$.
\end{exm}

\begin{rmk}
If a monoid $M$ happens to be a group, then there is a canonical equivalence $M \simeq M^{\text{op}}$ defined by the inverse homomorphism $m \mapsto m^{-1}$.  Our next few observations exploit an analogue of this equivalence for \textit{grouplike} $\mathbb{A}_\infty$-spaces.  We denote by \textbf{LoopSpaces} the full subcategory of grouplike objects in $\mathbb{A}_\infty(\textbf{Spaces})$.  Notice that the property of being grouplike is preserved under the \text{op} action on $\mathbb{A}^{\sigma}_\infty(\textbf{Spaces})$, so there is an \text{op} action on \textbf{LoopSpaces}.
\end{rmk}

\begin{cnstr}
There is a diagram of \textit{equivalences} of equivariant categories
$$
\begin{tikzcd}
& \textbf{ConnectedSpaces} \arrow[loop,distance=2em, swap]{}{trivial} \arrow[swap]{ld}{\Omega^{\sigma}} \arrow{rd}{\Omega} \\
\textbf{LoopSpaces} \arrow[out=-130, in=-50, loop, distance=2em, swap]{}{\text{op}} && \textbf{LoopSpaces} \arrow[out=-130, in=-50, loop, distance=2em, swap]{}{\text{trivial}}
\end{tikzcd}
$$
The equivariant functors $\Omega$ and $\Omega^{\sigma}$ share the same underlying, non-equivariant functor.
\end{cnstr}

\begin{proof}
It is classical that $\Omega$ and the bar construction provide inverse equivalences of the non-equivariant categories \textbf{ConnectedSpaces} and \textbf{LoopSpaces}.  This category has a universal property: it is the initial pointed category with all connected colimits.  As such, any $C_2$-action on it admits an essentially unique equivalence with the trivial $C_2$-action.
\end{proof}

\begin{cor} \label{DeloopOverCat}
Suppose $X$ is a grouplike $\mathbb{A}_\infty$-space with involution.  Then there exists some connected space with $C_2$-action $B^{\sigma}X$ such that $\Omega^{\sigma} B^{\sigma} X \simeq X$.  There is a natural $C_2$-equivariant functor
$$\textbf{Spaces}_{/B^{\sigma} X} \stackrel{\Omega^{\sigma}}{\longrightarrow} \mathbb{A}^{\sigma}_\infty\left(\textbf{Spaces}_{/X}\right),$$
where the latter object is the category with $C_2$-action underlying the monoidal category with involution from Example \ref{OverCat}.
\end{cor}

Consider the sequence of right adjoints
$$
\begin{tikzcd} [column sep=large]
\textbf{Spaces} \arrow{r}{(\--)_0} & \textbf{ConnectedSpaces} \arrow{r}{\Omega^{\sigma}} & \mathbb{A}^{\sigma}_\infty(\textbf{Spaces}) \arrow{r}{Forget} & \textbf{Spaces}
\end{tikzcd}
$$

By Example \ref{SymmetricMonoidalExample}, this is an equivariant functor from \textbf{Spaces} with trivial action to \textbf{Spaces} with trivial action.  As such it sends any space with $C_2$-action $X$ to some other space with $C_2$-action, which by abuse of notation we denote $\Omega^{\sigma} X$. 

\begin{prop} \label{OmegaSigmaIdentification}
Suppose $X$ is a space with $C_2$-action.  Then the space with $C_2$-action $\Omega^{\sigma} X$ is the equivariant function space $\Hom(S^{\sigma},X)$.  In other words, the action on a loop $S^1 \rightarrow X$ is given by both precomposing with the complex conjugation on $S^1$ and postcomposing with the action on $X$.
\end{prop}

\begin{proof}
By the Yoneda lemma, $\Omega^{\sigma} X$ must be the equivariant function spectrum $\Hom(S^1,X)$ for \textit{some} $C_2$-action on $S^1$.  To determine that this action is the complex conjugation action, and not the trivial action, we look at the sequence of equivariant functors 
$$
\begin{tikzcd}[column sep=large]
\textbf{Groups} \arrow{r}{Bar} & \textbf{ConnectedSpaces} \arrow{r}{\Omega^{\sigma}} & \textbf{LoopSpaces} \arrow{r}{\pi_0} & \textbf{Groups},
\end{tikzcd}
$$
which connects groups with trivial action to groups with $\text{op}$ action.  The only \textit{natural} isomorphism between a group and its opposite is given by $g \mapsto g^{-1}$, which is non-trivial on underlying sets.
\end{proof}

Specializing the discussion, recall that $MUP$ is the Thom spectrum of the $J$ homomorphism $BU \times \mathbb{Z} \stackrel{J}{\longrightarrow} BGL_1(\mathbb{S})$.  Since the $J$ homomorphism is an infinite loop map, $MUP$ acquires the structure of an $\mathbb{E}_\infty$-ring spectrum.  The complex-conjugation action by infinite loop maps on $BU \times \mathbb{Z}$ gives $MUP$ a $C_2$-action by $\mathbb{E}_\infty$-ring maps.  This in turn induces $C_2$-actions by symmetric monoidal functors on the categories $\text{MUP}\textbf{-Mod}$ and $\textbf{Spaces}_{/BGL_1(MUP)}$.  There is a diagram of lax symmetric monoidal functors
$$
\begin{tikzcd}
\textbf{Spaces}_{/BGL_1(MUP)} \arrow{r}{Thom} & \text{MUP}\textbf{-Mod} \arrow{d}{L_{K(n)}} \arrow{r}{Forget} & \textbf{Spectra} \arrow{d}{L_{K(n)}} \\
& \text{MUP}\textbf{-Mod} \arrow{r}{Forget} & \textbf{Spectra}.
\end{tikzcd}
$$
Since the $C_2$-action on $MUP$ is unital, the entire diagram becomes $C_2$-equivariant once we equip \textbf{Spectra} with the trivial $C_2$-action.  We may thus view the diagram as one of morphisms in the category of homotopy fixed points of $\textbf{SymMonCat}_{Lax}$ with trivial action.  This in turn induces a diagram in the homotopy fixed point category of $\textbf{MonCat}_{Lax}$ with \text{rev} action, which yields a diagram of $C_2$-equivariant categories
$$
\begin{tikzcd}
\mathbb{A}_\infty^{\sigma}\left(\textbf{Spaces}_{/BGL_1(MUP)}\right) \arrow{r}{Thom} & \mathbb{A}^{\sigma}_\infty(\text{MUP}\textbf{-Mod}) \arrow{r}{Forget} \arrow{d}{L_{K(n)}} & \mathbb{A}^{\sigma}_\infty(\textbf{Spectra}) \arrow{d}{L_{K(n)}} \\
& \mathbb{A}^{\sigma}_\infty(\text{MUP}\textbf{-Mod}) \arrow{r}{Forget} & \mathbb{A}^{\sigma}_{\infty}(\textbf{Spectra}).
\end{tikzcd}
$$
This is nearly all of our diagram (\ref{equidiagram}).  To complete the diagram, we use Corollary \ref{DeloopOverCat} for ${X \simeq BGL_1(MUP)}$ and Example \ref{SymmetricMonoidalExample} for $\mathcal{C}$ the symmetric monoidal category of \textbf{Spectra}.  

\begin{rmk}
In the sequel, we will denote the space with $C_2$-action $B^{\sigma}BGL_1(MU_\mathbb{R}P)$ by $B^{\rho}GL_1(MU_\mathbb{R}P)$.
\end{rmk}

\section{An Equivariant Map \texorpdfstring{\text{to }$B^{\rho}GL_1(MU_\mathbb{R}P)$}{}} \label{ThomSection}

The previous Section \ref{ConstructionSection} constructs a $C_2$-equivariant functor
$$\textbf{Spaces}_{/B^{\rho} GL_1(MUP)} \rightarrow \mathbb{A}^{\sigma}_\infty(\textbf{Spectra}),$$
which remains $C_2$-equivariant after composing with $K(n)$-localization.  Here, $\textbf{Spaces}_{/B^{\rho}GL_1(MUP)}$ is granted its $C_2$-action via the one on the space $B^{\rho}GL_1(MUP)=B^{\rho}GL_1(MU_{\mathbb{R}}P)$.  The category $\mathbb{A}^{\sigma}_\infty(\textbf{Spectra})$ is equipped with the $\text{op}$ action of Example \ref{SymmetricMonoidalExample}.

A homotopy fixed point for $\textbf{Spaces}_{/B^{\rho}GL_1(MUP)}$ is just a map of spaces with $C_2$-action $X \rightarrow B^{\rho} GL_1(MUP)$, and such a map therefore gives rise to a homotopy fixed point of $\mathbb{A}^{\sigma}_\infty(\textbf{Spectra})$.  In other words, an equivariant map of spaces with $C_2$-action $f:X \rightarrow B^{\rho} GL_1(MUP)$ gives rise to a Thom $\mathbb{A}_\infty$-algebra with involution $(\Omega^{\sigma}X)^{\Omega^{\sigma}f}.$

One can apply the construction to $* \rightarrow B^{\rho}GL_1(MUP)$ to obtain $MUP$ itself as an $\mathbb{A}_{\infty}$-algebra with involution.  Using the equivariant map $* \rightarrow X$, we obtain a map of $\mathbb{A}_\infty$-rings with involution $MUP \rightarrow (\Omega^{\sigma}X)^{\Omega^{\sigma}f}$.  The canonical Real orientation $$\Sigma^{-2}\mathbb{CP}^{\infty} \rightarrow MU_{\mathbb{R}} \rightarrow MU_\mathbb{R}P$$
then equips $(\Omega^{\sigma} X)^{\Omega^{\sigma}f}$ with a Real orientation.  If $(\Omega^{\sigma} X)^{\Omega^{\sigma} f}$ happens to also be a $C_2$-equivariant homotopy commutative ring, then \cite[Theorem 2.25]{HuKriz} implies that it receives an equivariant homotopy commutative ring map from $MU_{\mathbb{R}}$.

In this section we will be concerned with the construction of a particular map of spaces with $C_2$-action into $B^{\rho}GL_1(MU_\mathbb{R}P)$; the underlying map of spaces will be the morphism
$$BS^3 \times BS^3 \times \cdots \rightarrow B^2GL_1(MUP)$$
constructed in Section \ref{ClassicalConstruction}.  Our aim is to construct both $2$-periodic Johnson--Wilson theory and Morava $E$-theory as $\mathbb{A}_\infty$-rings with involution.

\begin{rmk}
Recall that, among spaces with $C_2$-action, we may identify certain representation spheres $S^{a+b\sigma}$ as the one-point compactifications of real $C_2$-representations.  We use $\sigma$ to denote the sign representation, $1$ to denote the trivial representation, and the shorthand $\rho$ to denote the regular representation $1+\sigma$.  If $X$ is a space or spectrum with $C_2$-action, then we use $\pi_{a+b\sigma}(X)$ to denote $\pi_0$ of the space of equivariant maps
$S^{a+b\sigma} \rightarrow X$.  Of interest to us, Proposition \ref{OmegaSigmaIdentification} implies that $\pi_{a+b\sigma}(B^{\sigma} X) \cong \pi_{a+(b-1)\sigma}(X).$
\end{rmk}

In \cite{HuKriz}, the equivariant homotopy groups of $MU_{\mathbb{R}}P$ are computed.  For each $n$, $\pi_{n\rho-1}(MU_\mathbb{R}P) \cong 0$.  Additionally, there is a ring isomorphism
$$\pi_{*\rho}(MU_\mathbb{R}P) \cong \mathbb{Z}[\bar{x}_1,\bar{x}_2,\cdots][\bar{u}^{\pm}],$$
where $\bar{x}_i$ is in degree $i\rho$ and $\bar{u}$ is in degree $\rho$.  The forgetful map from the equivariant to ordinary homotopy groups
$\pi_{*\rho}(MU_\mathbb{R}P) \rightarrow \pi_{2*}(MUP)$ takes $\bar{x}_i$ to $x_i$ and $\bar{u}$ to $u$.

Since $GL_1(MU_\mathbb{R}P)$ is defined via a pullback square of spaces with $C_2$-action
$$
\begin{tikzcd}
GL_1(MU_\mathbb{R}P) \arrow{r} \arrow{d} & \Omega^{\infty} MUP \arrow{d} \\
\pi_0(MU_\mathbb{R}P)^{\times} \arrow{r} & \pi_0(MU_\mathbb{R}P),
\end{tikzcd}
$$
we learn that $\pi_{a+b \sigma}(B^{\rho}GL_1(MU_\mathbb{R}P)) \cong \pi_{(a-1)+(b-1)\sigma}(MU_\mathbb{R}P)$ whenever $a,b>1$.

Our next task is to understand the $C_2$-equivariant space $B^{\sigma}S^{\rho+1}$.  The analogue of the even cell structure that played a prominent role in Section \ref{ClassicalConstruction} is the following:

\begin{prop}
There is a $C_2$-action on $\mathbb{HP}^\infty$ so that $\mathbb{HP}^\infty$ arises as a filtered colimit 
$$Y_1=S^{2\rho} \rightarrow Y_2 \rightarrow Y_3 \rightarrow \cdots,$$
where there are homotopy pushout square of spaces with $C_2$-action
$$
\begin{tikzcd}
S^{2n\rho-1} \arrow{d} \arrow{r} & Y_{n-1} \arrow{d} \\
* \arrow{r} & Y_{n}.
\end{tikzcd}
$$
Furthermore, $\Omega^\sigma \mathbb{HP}^\infty \simeq S^{\rho+1}$.  We will therefore write $B^\sigma S^{\rho+1}$ to denote $\mathbb{HP}^\infty$ with this $C_2$-action. 
\end{prop}

\begin{proof}
This cell decomposition is due to Mike Hopkins.  Recall that, non-equivariantly, the cellular filtration on $BS^3$ agrees with the standard filtration
$$\mathbb{HP}^1 \rightarrow \mathbb{HP}^2 \rightarrow \cdots \rightarrow \mathbb{HP}^{\infty} \simeq BS^3,$$
where $\mathbb{HP}^{\infty}$ is the infinite-dimensional quaternionic projective space.  For us, the relevant $C_2$-action on this space is conjugation by $i$.  In other words, we act on a point
$$[z_0:z_1:z_2:\cdots]$$
by sending it to $$[iz_0i^{-1}:iz_1i^{-1}:\cdots].$$  From the expression $i (a+bi+cj+dk) i^{-1} = a+bi - cj -dk$ we learn both that the action is well-defined and that the $C_2$-cells attached are multiples of $2 \rho$.  Furthermore, with this $C_2$-action, the $C_2$-fixed points of $\mathbb{HP}^\infty$ are $\mathbb{CP}^\infty$ and the map ${\mathbb{CP}^\infty = (\mathbb{HP}^\infty)^{C_2} \to \mathbb{HP}^\infty}$ is the usual inclusion map.  

The non-equivariant map $S^4 \rightarrow BS^3 \simeq \mathbb{HP}^{\infty}$ adjoint to the identity is lifted to an equivariant map $S^{2\rho} \rightarrow \mathbb{HP}^{\infty}$, given by the inclusion $\mathbb{HP}^1 \rightarrow \mathbb{HP}^{\infty}$ under the described $C_2$-action.  By adjunction, this gives a map $S^{\rho+1} \to \Omega^\sigma \mathbb{HP}^\infty$ of $C_2$-spaces which is an underlying equivalence.  We will show that this map induces a $C_2$-equivalence by checking that it is an equivalence on $C_2$-fixed points.

The map $S^{\rho+1} \to \Omega^\sigma \mathbb{HP}^\infty$ is the composite map 
$$S^{\rho+1} \longrightarrow \Omega^{\sigma} \Sigma^\sigma S^{\rho+1} \longrightarrow \Omega^\sigma \mathbb{HP}^\infty,$$
where the first map is the unit map of the loop-suspension adjunction and the second map is obtained by applying $\Omega^\sigma(-)$ to the inclusion $S^{2\rho} \to \mathbb{HP}^\infty$.  On $C_2$-fixed points, we have the commutative diagram 
$$\begin{tikzcd}
(S^{\rho+1})^{C_2} = S^2 \ar[r, "\textbf{1}"] & (\Omega^\sigma \Sigma^\sigma S^{\rho+1})^{C_2} \ar[r, "\textbf{2}"] \ar[d, "\textbf{3}"]  & (\Omega^\sigma \mathbb{HP}^\infty)^{C_2} \simeq S^2 \ar[d, "\textbf{4}"] \\ 
&(\Sigma^\sigma S^{\rho+1})^{C_2} = S^2 \ar[r, "\textbf{5}"] \ar[d] &(\mathbb{HP}^\infty)^{C_2} = \mathbb{CP}^\infty \ar[d] \\ 
&S^4 \ar[r] & \mathbb{HP}^\infty.
\end{tikzcd}$$
The two vertical sequences are fiber sequences obtained by first mapping the cofiber sequence ${{C_2}_+ \to S^0 \to S^\sigma}$ to $\Sigma^\sigma S^{\rho+1}$ and $\mathbb{HP}^\infty$, respectively, and then taking $C_2$-fixed points.  Since the map $(\mathbb{HP}^\infty)^{C_2} \to \mathbb{HP}^\infty$ is the map $BS^1 \to BS^3$, its fiber $(\Omega^\sigma \mathbb{HP}^\infty)^{C_2}$ is $S^3/S^1 \simeq S^2$.  
Maps $\textbf{4}$ and $\textbf{5}$ are both the usual inclusion map $S^2 = \mathbb{CP}^1 \to \mathbb{CP}^\infty$.  

For any $C_2$-space $X$, consider the composite map
$$X \longrightarrow \Omega^\sigma \Sigma^\sigma X \longrightarrow  \Sigma^\sigma X,$$
where the first map is the unit map of the loop-suspension adjunction and the second map is obtained by mapping out of the inclusion $S^0 \to S^\sigma$ of $C_2$-spaces.  On $C_2$-fixed points, this map induces an equivalence $X^{C_2} \simeq X^{C_2} = (\Sigma^\sigma X)^{C_2} $.  This shows that the composite map $\textbf{3} \circ \textbf{1}$ is an equivalence.  It follows from the commutativity of the diagram that the composite map $\textbf{2} \circ \textbf{1}$ is also an equivalence.
\end{proof}

As a corollary, exactly as in Proposition~\ref{ClassicLoop}, we learn that any map $S^{2\rho} \rightarrow B^{\rho}GL_1(MU_\mathbb{R}P)$ factors through $B^{\sigma} S^{\rho+1}$.  This is because the map $S^{2\rho} \to B^\rho GL_1(MU_\mathbb{R}P)$ can be viewed as a map out of $Y_1$, and the obstruction to factor a map $Y_{n-1} \to B^{\rho} GL_1 (MU_{\mathbb{R}}P)$ through $Y_n$ lives in $\pi_{2n\rho-1}(B^\rho GL_1 MU_\mathbb{R}P) \cong \pi_{(2n-1)\rho-1} MU_\mathbb{R}P$, which is zero.  

Using the symmetric monoidal structure on $\textbf{Spaces}_{/B^{\rho}GL_1(MUP)}$, which commutes with the $C_2$-action, we may construct from any sequence $(\alpha_1,\alpha_2,\cdots) \in \pi_{\rho} MU_\mathbb{R}P$ a map
$$S^{\rho+1} \times S^{\rho+1} \times \cdots \rightarrow BGL_1(MU_\mathbb{R}P).$$
This then factors through at least one equivariant map
$$B^{\sigma}S^{\rho+1} \times B^{\sigma}S^{\rho+1} \times \cdots \rightarrow B^{\rho}GL_1(MU_\mathbb{R}P).$$

We choose for $(\alpha_0,\alpha_1,\cdots)$ the same sequence as in Lemma \ref{fglLemma}, with the $x_i$ replaced by $\bar{x}_i$.  The reader may prefer to consider the special case in which the sequence is $$(\bar{x}_{2^n-1}\bar{u}^{2-2^n}-\bar{u}, \bar{x}_2\bar{u}^{-1},\bar{x}_4\bar{u}^{-3},\bar{x}_5\bar{u}^{-4},\bar{x}_6\bar{u}^{-5},\bar{x}_8\bar{u}^{-7}, \cdots),$$ where the classes $\bar{x}_i\bar{u}^{-i+1}$ that are included in the sequence are all those such that either
\begin{itemize}
\item $i$ is not one less than a power of $2$.
\item $i$ is greater than $2^n-1$.
\end{itemize}

In any case, applying $\Omega^{\sigma}$ and then the Thom construction we obtain a homotopy fixed point of the category $\mathbb{A}^{\sigma}_\infty(MUP\textbf{-Mod}).$  The underlying $\mathbb{A}_\infty$-ring is $E(n)$, the $2$-periodic version of Johnson--Wilson theory constructed in Section \ref{ClassicalConstruction}.  Our constructions produce a coherent $\mathbb{A}_\infty$-ring map $E(n) \stackrel{\simeq}{\longrightarrow} E(n)^{\text{op}}$ lifting the complex-conjugation $C_2$-action $E(n) \longrightarrow E(n)$.  We denote this ring with involution by $E_\mathbb{R}(n)$.

\begin{rmk}
From this work, it seems that the natural action on $E(n)$ is by $\mathbb{A}_\infty$-involutions rather than $\mathbb{A}_\infty$-algebra maps.  However, we can sketch an approach to producing an action by $\mathbb{A}_\infty$-algebra maps in the spirit of this paper.

Using Theorem $1.4$ of \cite{ThomQuotient}, $E(n)$ can be built as a Thom spectrum of a map $SU \rightarrow BGL_1(MUP)$.  Obstruction theory easily lifts this to a map $BSU \rightarrow B^2GL_1(MUP)$, which produces the same involution we see above.  We may go further though, and note that $B^3SU$ also has an even cell structure.  This means that it is easy to produce maps $B^3 SU \rightarrow B^4 GL_1(MUP)$, but as noted in \cite[\S 6]{ChadwickMandell} it is not so easy to know which maps $BSU \rightarrow B^2GL_1(MUP)$ these lie over.  If one could produce $E(n)$ as a Thom $\mathbb{E}_3$-$MUP$-algebra in this way, non-equivariantly, it seems likely that one could produce an $\mathbb{E}_{2\sigma+1}$-structure on the equivariant $E(n)$.  In particular, this would mean the $C_2$-action on $E(n)$ is by $\mathbb{A}_\infty$-ring homomorphisms.

This may be of interest in light of \cite{PhantomRing}, in which Kitchloo, Lorman, and Wilson provide a homotopy commutative and associative ring structure \textit{up to phantom maps} on Real Johnson-Wilson theory.  We thank Kitchloo for pointing out to us that the difficulty with phantom maps disappears after $K(n)$-localization.
\end{rmk}

\begin{rmk}
Our arguments also show that the Real Morava $K$-theories $K_{\mathbb{R}}(n)$ \cite[Section 3]{HuKriz} have the structure of rings with involution.  More precisely, they are $\mathbb{E}_{\sigma}$-$MU_{\mathbb{R}}P$-algebras.
\end{rmk}

\section{Proof of Theorem~\ref{FullVersion}} \label{sec:Thm1.1}

In the previous section, we constructed an $\mathbb{A}_\infty$-ring spectrum $E(n)$ with a $C_2$-action by $\mathbb{A}_\infty$-involutions.  After $K(n)$-localizing, we obtain a $C_2$-action by involutions on Morava $E$-theory $\widehat{E(n)}$.

Now, consider the equivariant sequence of forgetful functors
$$\mathbb{E}_\infty(\textbf{Spectra}) \rightarrow \mathbb{A}^{\sigma}_\infty(\textbf{Spectra}) \rightarrow \textbf{Spectra},$$
where both $\mathbb{E}_\infty(\textbf{Spectra})$ and $\textbf{Spectra}$ are given the trivial $C_2$-action, but $\mathbb{A}^{\sigma}_\infty(\textbf{Spectra})$ is given the $\text{op}$ action.  We may restrict this sequence to an equivariant sequence of subcategories
$$\mathcal{C}_3 \rightarrow \mathcal{C}_2 \rightarrow \mathcal{C}_1,$$
where
\begin{itemize}
\item $\mathcal{C}_1$ is the category of all spectra equivalent to $\widehat{E(n)}$ and equivalences between them.
\item $\mathcal{C}_2$ is the category of $\mathbb{A}_\infty$-ring spectra with underlying spectrum $\widehat{E(n)}$, and equivalences between them.
\item $\mathcal{C}_3$ is the category of $\mathbb{E}_\infty$-ring spectra with underlying spectrum $\widehat{E(n)}$, and equivalences between them.
\end{itemize}

Note that a map of categories with $C_2$-action is equivalence if and only if the underlying non-equivariant functor is an equivalence of non-equivariant categories.  The Goerss--Hopkins--Miller theorem \cite{GoerssHopkins, HopkinsMiller} says that the map $\mathcal{C}_3 \rightarrow \mathcal{C}_2$ is an equivalence of categories.  It follows that any homotopy fixed point of $\mathcal{C}_2$ may uniquely be lifted to one of $\mathcal{C}_3$.  Thus, the $C_2$-action on $\widehat{E(n)}$ by $\mathbb{A}_\infty$-involutions has a unique lift to a $C_2$-action by $\mathbb{E}_\infty$-ring automorphisms.  According to Goerss--Hopkins--Miller \cite{HopkinsMiller}, the categories $\mathcal{C}_3$ and $\mathcal{C}_2$ are equivalent to $B\mathbb{G}_n$, where $\mathbb{G}_n$ is the Morava stabilizer group.  A $C_2$-action on $\widehat{E(n)}$ by $\mathbb{E}_\infty$-ring maps is therefore the data of a map $BC_2 \rightarrow B\mathbb{G}_n$, which is the data of a group homomorphism $C_2 \rightarrow \mathbb{G}_n$.  It follows by direct calculation that any $C_2$-action by $\mathbb{E}_\infty$-ring maps is determined by its effect on homotopy groups.  The Real orientation $MU_\mathbb{R} \rightarrow E_\mathbb{R}(n) \rightarrow \widehat{E(n)}$ determines that the $C_2$-action we have constructed is the one that acts by the formal inverse, proving Theorem \ref{MainTheorem}.

\begin{rmk}
Our discussion of algebras with involution, and our use of the Goerss--Hopkins--Miller Theorem, may both be entirely avoided if one only wants to know that the $C_2$-action on $\widehat{E(n)}$ is the Galois one \textit{in the homotopy category of spectra}.  It is, however, not a priori clear that there is a unique lift of this homotopy $C_2$-action on $\widehat{E(n)}$ to a fully coherent $C_2$-action.
\end{rmk}

To prove Theorem \ref{FullVersion}, we need to prove that there is a Real orientation of the Lubin--Tate theory $E_{(k,\Gamma)}$ associated to any finite height formal group law $\Gamma$ over a perfect field $k$ of characteristic $2$.  If the formal group law $\Gamma$ is defined over $\mathbb{F}_2$, then we may use a Real orientation of some $\widehat{E(n)}$ in order to orient $E_{(k,\Gamma)}$.  In general, there will be a map of Lubin--Tate theories $E_{(k,\Gamma)} \longrightarrow E_{(\bar{k},\overline{\Gamma})}$, where $\bar{k}$ is the algebraic closure of $k$ and $\overline{\Gamma}$ is the pushforward of $\Gamma$ to this algebraic closure.  Since any two height $n$ formal groups over an algebraically closed field are isomorphic, $\overline{\Gamma}$ is isomorphic to a Honda formal group defined over $\mathbb{F}_2$.  We will see in Section~\ref{subsec:HFPSSEnHC2} that this forces the $C_2$-equivariant homotopy fixed point spectral sequence of $E_{(\bar{k}, \overline{\Gamma})}$ to be regular in the sense of \cite[Definition~6.1]{MeierTMF}.  In particular, $\pi^{C_2}_{*\rho-1} E_{(\bar{k}, \overline{\Gamma})} \cong 0$ and $\pi^{C_2}_{*\rho} E_{(\bar{k},\overline{\Gamma})}$ is a copy of the non-equivariant homotopy groups $\pi_{2*} E_{(\overline{k},\overline{\Gamma})}$.  The latter statement implies that $E_{(\bar{k},\overline{\Gamma})}$ is a free module over $E_{(k,\Gamma)}$, and so in particular $E_{(k,\Gamma)}$ is a retract of $E_{(\bar{k},\overline{\Gamma})}$.  The former statement then tells us that $\pi^{C_2}_{*\rho-1} E_{(k,\Gamma)} \cong 0$, so that there are no obstructions to making a Real orientation of $E_{(k,\Gamma)}$ (\cite[Lemma~3.3]{HillMeier}).

\begin{rmk}
Since writing the first version of this paper, Lennart Meier has very cleanly formulated \cite{MeierTMF} the notion of an even ring spectrum having a \textit{regular homotopy fixed point spectral sequence.}  He generously attributes \cite[Example 6.11]{MeierTMF} to this work the theorem that Lubin--Tate $E_{(k,\Gamma)}$ have regular homotopy fixed point spectral sequences, but in fact we do not quite prove this here.  Rather, we prove in Section~\ref{subsec:HFPSSEnHC2} only that $E_{(k,\Gamma)}$ has regular $C_2$-HFPSS when $\Gamma$ is pushed forward from a formal group law defined over $\mathbb{F}_2$.  Nonetheless, Meier's Proposition 6.5 in \cite{MeierTMF} applies to show that, since the $C_2$-HFPSS for $E_{(\bar{k},\overline{\Gamma})}$ is regular, so must be the $C_2$-HFPSS for $E_{(k,\Gamma)}$.  Thus, it is really the combination of our work here with \cite[Proposition~6.5]{MeierTMF} that implies the regularity of the $C_2$-HFPSS for all Lubin--Tate theories.
\end{rmk}

If $G$ is a finite subgroup of the $\mathbb{E}_\infty$-ring automorphisms of $E_{(k,\Gamma)}$ containing the central $C_2$, there then arises a sequence of homotopy ring maps
$$N^{G}_{C_2} MU_{\mathbb{R}} \longrightarrow N^{G}_{C_2} E_{(k,\Gamma)} \longrightarrow E_{(k,\Gamma)}.$$
The existence of the last homomorphism follows from the fact that the norm is an adjunction between $\mathbb{E}_\infty$-rings with $C_2$-action and $\mathbb{E}_\infty$-rings with $G$-action (see \cite[\S 2.2]{HillMeier}).

\section{Real Landweber Exactness and Proof of Theorem~\ref{FullVersion2}} \label{sec:HFPSSEn}

In the remainder of the paper, for simplicity, we use a specific Morava $E$-theory $E_n$ that is defined via a lift of the height $n$ Honda formal group law over $\mathbb{F}_{2^n}$.  Its homotopy groups are 
$$\pi_*E_n = W(\mathbb{F}_{2^n})[[u_1, u_2, \ldots, u_{n-1}]][u^\pm].$$
and the $2$-typical formal group law over $\pi_* E_n$ is determined by the map $\pi_*BP \to \pi_*{E_n}$ sending 
$$v_i \mapsto \left \{ \begin{array}{ll} u_i u^{2^i -1} & 1 \leq i \leq n-1 \\ 
u^{2^n-1} & i = n \\ 
0 & i > n.
\end{array} \right.$$
Our results are all easily generalized to other variants of Morava $E$-theory.

In this section, we will show that $E_n$, as a $C_2$-spectrum, is Real Landweber exact in the sense of \cite{HillMeier}.  We do so by completely computing the $RO(C_2)$-graded homotopy fixed point spectral sequence of $E_n$.  

\subsection{$RO(C_2)$-graded homotopy fixed point spectral sequence of \texorpdfstring{$E_n$}{En}} \label{subsec:HFPSSEnHC2}

So far, we have constructed a $C_2$-equivariant map from 
$$MU_\mathbb{R} \to E_n.$$  
Here, the $C_2$-action on $MU_\mathbb{R}$ is by complex conjugation, and the $C_2$-action on $E_n$ is by the Goerss--Hopkins--Miller $E_\infty$-action.  The existence of this equivariant map will help us in computing the $C_2$-homotopy fixed point spectral sequence of $E_n$.  In particular, the map $MU_\mathbb{R} \to E_n$ induces the map of spectral sequences
$$C_2\text{-}\HFPSS(MU_\mathbb{R}) \to C_2\text{-}\HFPSS(E_n)$$ 
of $C_2$-equivariant homotopy fixed point spectral sequences.  Since both the complex conjugation action on $MU_\mathbb{R}$ and the Galois $C_2$-action on $E_n$ are by $E_\infty$-ring maps, both spectral sequences are multiplicative (the map between them is not necessarily a multiplicative map, but this is perfectly fine).  At this point, we will replace $MU_\mathbb{R}$ by $BP_\mathbb{R}$ because everything is 2-local, and argument below is exactly the same regardless of whether we are using $MU_\mathbb{R}$ or $BP_\mathbb{R}$.  Moreover, since $MU_\mathbb{R}$ splits as a wedge of suspensions of $BP_\mathbb{R}$'s, the homotopy fixed point spectral sequence of $BP_\mathbb{R}$ has the advantage of having fewer classes than $MU_\mathbb{R}$ while still retaining the important 2-local information that we need. 

By \cite[Corollary~4.7]{HillMeier}, the $E_2$-pages of the $RO(C_2)$-graded homotopy fixed point spectral sequences of $BP_\mathbb{R}$ and $E_n$ are 
\begin{eqnarray} 
E_2^{s, t}(BP_\mathbb{R}^{hC_2}) &=& \mathbb{Z}[\bar{v}_1, \bar{v}_2, \ldots] \otimes \mathbb{Z}[u_{2\sigma}^\pm, a_\sigma]/(2a_\sigma) \nonumber \\ 
E_2^{s, t} (E_n^{hC_2}) &=& W(\mathbb{F}_{2^n})[[\bar{u}_1, \bar{u}_2, \ldots, \bar{u}_{n-1}]][\bar{u}^\pm] \otimes \mathbb{Z}[u_{2\sigma}^\pm, a_\sigma]/(2a_\sigma). \nonumber 
\end{eqnarray}

On the $E_2$-page, the class $\bar{v}_i$ is in stem $|\bar{v}_i| = i\rho$ for $i \geq 1$; the class $\bar{u}_i$ is in stem $|\bar{u}_i| = 0$ for $1 \leq i \leq n-1$; and the class $\bar{u}$ is in stem $|\bar{u} | =\rho$.  The classes $u_{2\sigma}$ and $a_\sigma$ are in stems $2 - 2\sigma$ and $-\sigma$, respectively.  They can be defined more generally as follows:

\begin{dfn} [$a_V$ and $u_V$]
Let $V$ be a representation of $G$ of dimension $d$.  
\begin{enumerate}
\item $a_V \in \pi_{-V}^G S^0$ is the map corresponding to the inclusion $S^0 \hookrightarrow S^V$ induced by $\{0\} \subset V$.
\item If $V$ is oriented, $u_V \in \pi_{d-V}^G \HZ$ is the class of the generator of $H_d^{G}(S^V; \HZ)$.  
\end{enumerate}
\end{dfn}

\noindent\textit{Proof of Theorem~\ref{FullVersion2}}.  In \cite{HuKriz}, Hu and Kriz completely computed the $C_2$-homotopy fixed point spectral sequence of $MU_\mathbb{R}$ and $BP_\mathbb{R}$.  In particular, the classes $\bar{v}_i$ for all $i \geq 1$ and the class $a_\sigma$ are permanent cycles.  All the differentials are determined by the differentials 
$$d_{2^{k+1} -1} (u_{2\sigma}^{2^{k-1}}) =  \bar{v}_ka_\sigma^{2^{k+1}-1}, \, k \geq 1$$
and multiplicative structures (see Figures~\ref{fig:BPRKRSSS}, \ref{fig:BPRER(2)SSS}, and \ref{fig:BPRER(3)SSS}).  There are no nontrivial extension problems on the $E_\infty$-page.\\

\begin{figure}
\begin{center}
\makebox[\textwidth]{\hspace{-0.5in}\includegraphics[trim={3cm 9.5cm 4.5cm 2cm}, clip, scale = 0.9]{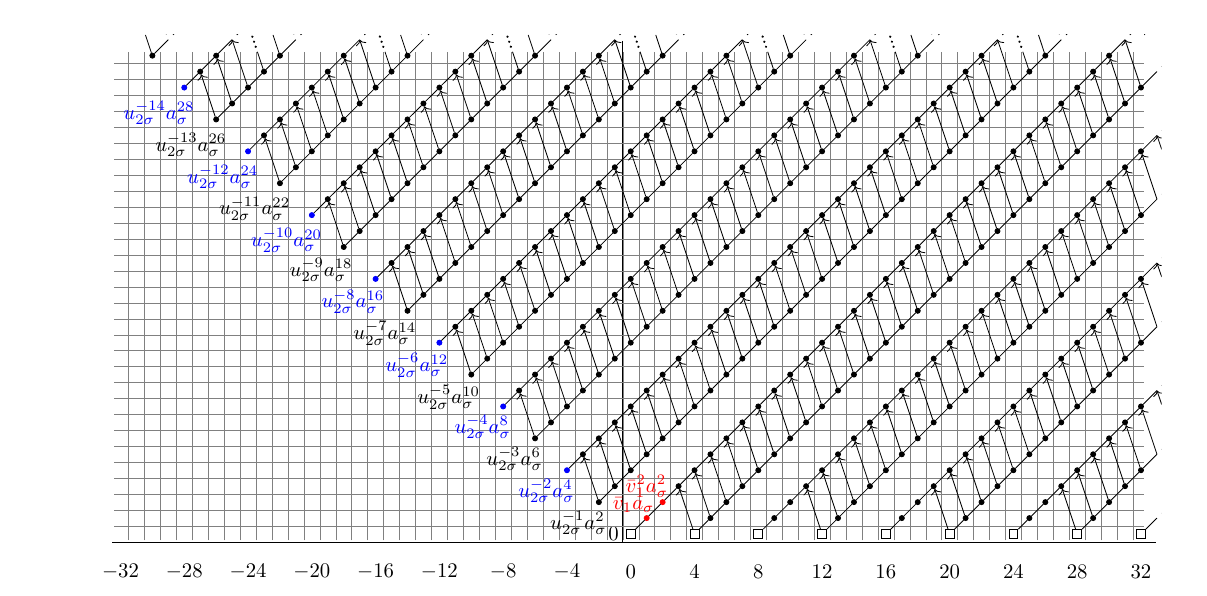}}
\end{center}
\begin{center}
\caption{Important $d_3$-differentials and surviving torsion classes on the $E_3$-page.}
\label{fig:BPRKRSSS}
\end{center}
\end{figure}

\begin{figure}
\begin{center}
\makebox[\textwidth]{\hspace{1in}\includegraphics[trim={5cm 7cm 0cm 2cm},clip, scale = 0.8]{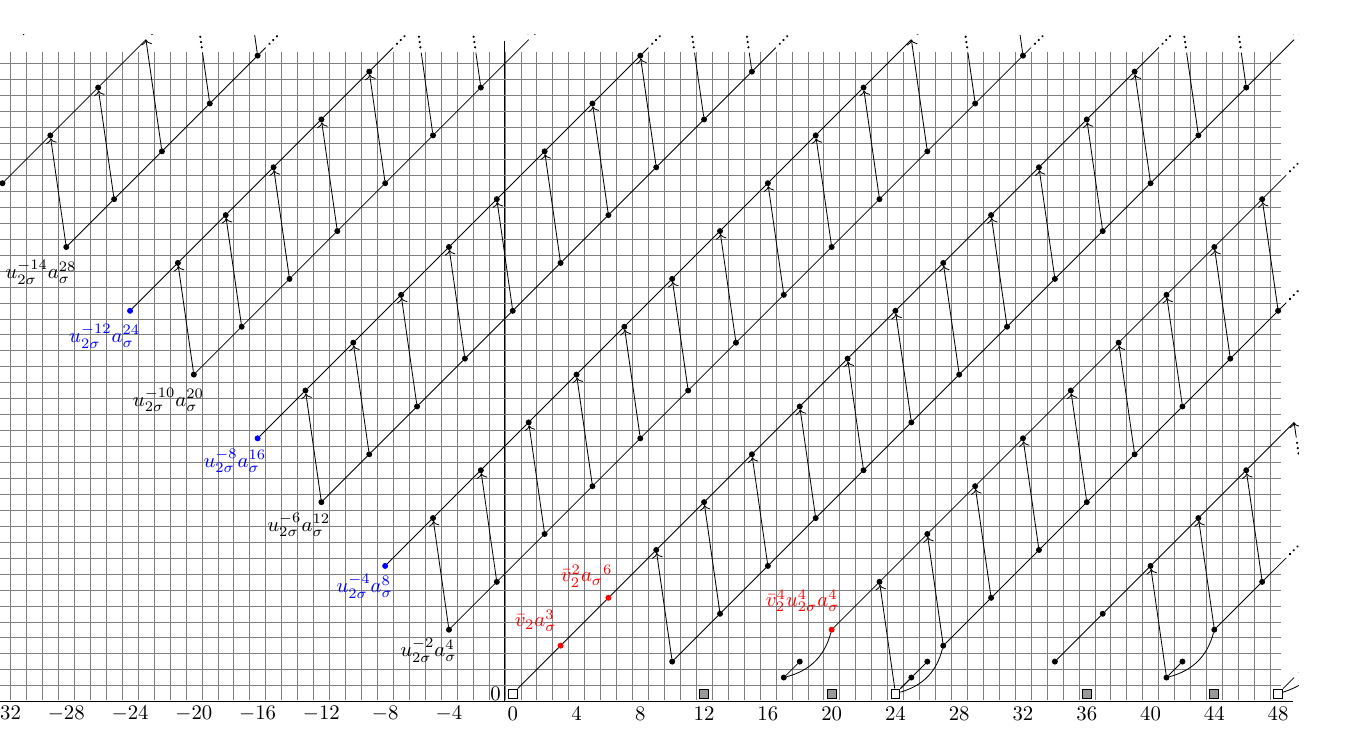}}
\end{center}
\begin{center}
\caption{Important $d_7$-differentials and surviving torsion classes on the $E_7$-page.}
\label{fig:BPRER(2)SSS}
\end{center}
\end{figure}

\begin{figure}
\begin{center}
\makebox[\textwidth]{\hspace{-0.7in}\includegraphics[trim={2cm 4.5cm 0cm 2cm},clip, scale = 0.8]{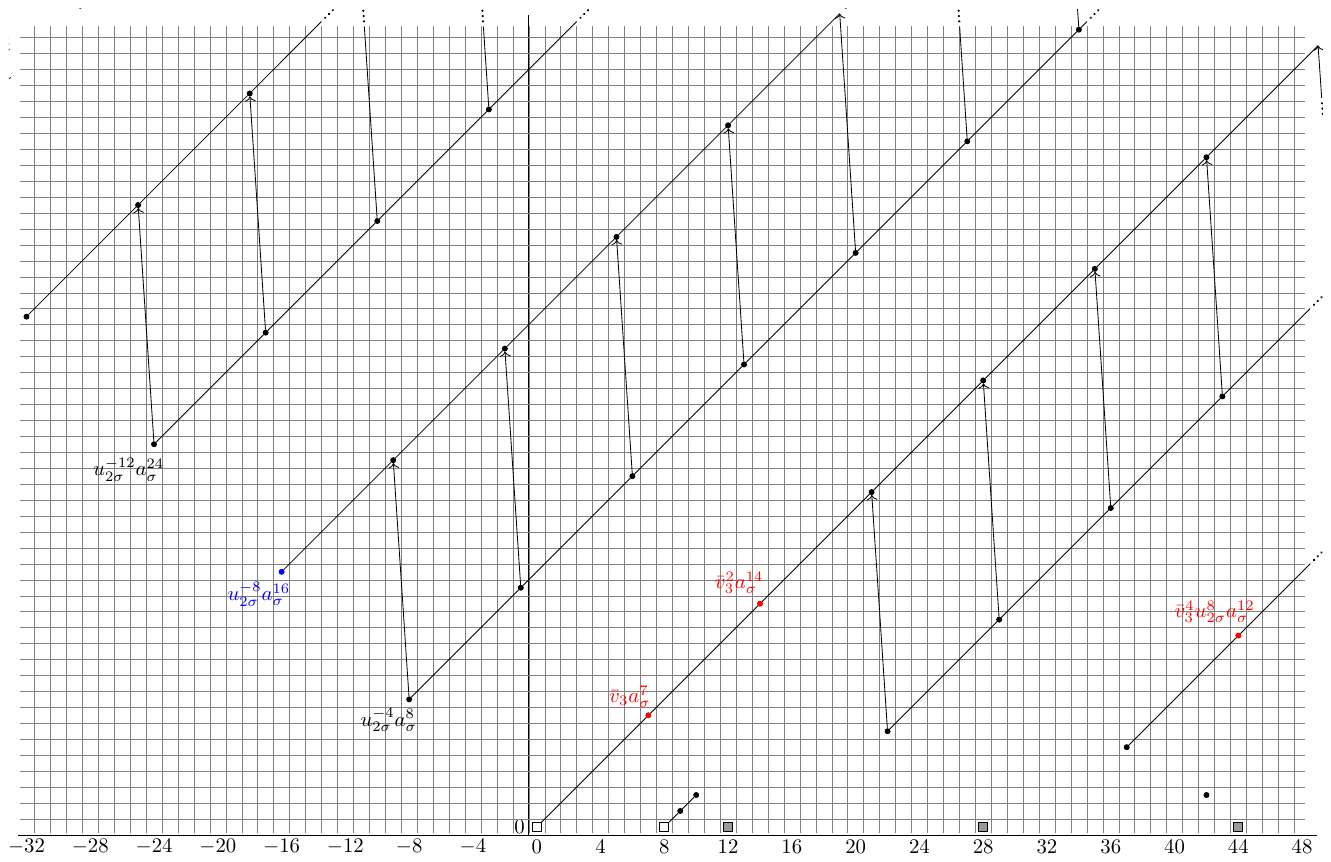}}
\end{center}
\begin{center}
\caption{Important $d_{15}$-differentials and surviving torsion classes on the $E_{15}$-page.}
\label{fig:BPRER(3)SSS}
\end{center}
\end{figure}


On the $E_2$-page, the map
$$C_2\text{-}\HFPSS(BP_\mathbb{R}) \to C_2\text{-}\HFPSS(E_n)$$ 
of spectral sequences sends the classes $u_{2\sigma} \mapsto u_{2\sigma}$, $\as \mapsto \as$, and 
$$\bar{v}_i \mapsto \left \{ \begin{array}{ll} \bar{u}_i \bar{u}^{2^i -1} & 1 \leq i \leq n-1 \\ 
\bar{u}^{2^n-1} & i = n \\ 
0 & i > n.
\end{array} \right.$$

We will first prove that the classes $\bar{u}_1$, $\ldots$, $\bar{u}_{n-1}$, $\bar{u}^\pm$, and $a_\sigma$ are permanent cycles in $C_2\text{-}\HFPSS(E_n)$.  Since the classes $\bar{v}_i$, $i \geq 1$, and $a_\sigma$ are permanent cycles in $C_2\text{-}\HFPSS(BP_\mathbb{R})$, their images are also permanent cycles in $C_2\text{-}\HFPSS(E_n)$.  This shows that the classes $\bar{u}_i \bar{u}^{2^i-1}$, $1 \leq i \leq n-1$, $\bar{u}^{2^n-1}$, and $a_\sigma$ are permanent cycles in $C_2\text{-}\HFPSS(E_n)$.  

Now, consider the non-equivariant map 
$$u: S^2 \to i_{e}^*E_n.$$  
Applying the Hill--Hopkins--Ravenel norm functor $N_e^{C_2}(-)$ (\cite{HHR}) produces the equivariant map 
$$N_e^{C_2}(u) = \bar{u}^2: S^{2\rho} \to N_{e}^{C_2} i_e^* E_n \to E_n,$$
where the last map is the co-unit map of the norm--restriction adjunction 
$$
N_e^{C_2}: \text{Commutative }C_2\text{-spectra} \rightleftarrows \text{Commutative spectra}: i_e^*.
$$
Since the element $N_e^{C_2}(u) = \bar{u}^2$ is an actual element in $\pi_\bigstar^{C_2}E_n$, it is a permanent cycle.  This, combined with the fact that $\bar{u}^{2^n-1}$ is a permanent cycle, shows that $\bar{u} = \bar{u}^{2^n-1} \cdot (\bar{u}^{-2})^{2^{n-1}}$ is a permanent cycle.  It follows from the previous paragraph that the classes $\bar{u}_1$, $\ldots$, $\bar{u}_{n-1}$, and $\bar{u}^\pm$ are all permanent cycles in $C_2\text{-}\HFPSS(E_n)$.  \\

It remains to produce the differentials in $C_2\text{-}\HFPSS(E_n)$.  We will show by induction on $k$, $1 \leq k \leq n$, that all the differentials in $C_2\text{-}\HFPSS(E_n)$ are determined by the differentials 
\begin{eqnarray*}
d_{2^{k+1} -1} (u_{2\sigma}^{2^{k-1}}) &=&  \bar{u}_k\bar{u}^{2^k-1}a_\sigma^{2^{k+1}-1}, \, \, \, 1 \leq k \leq n-1, \\ 
d_{2^{n+1}-1}(u_{2\sigma}^{2^{n-1}})&=& \bar{u}^{2^n-1}a_\sigma^{2^{n+1}-1}, \, \, \, k = n 
\end{eqnarray*}
and multiplicative structures. 

For the base case, when $k=1$, there is a $d_3$-differential 
$$d_3(u_{2\sigma}) = \bar{v}_1a_\sigma^3$$
in $C_2\text{-}\HFPSS(BP_\mathbb{R})$.  Under the map 
$$ C_2\text{-}\HFPSS(BP_\mathbb{R}) \to C_2\text{-}\HFPSS(E_n)$$
of spectral sequences, the the source is mapped to $u_{2\sigma}$ and the target is mapped to $\bar{u}_1\bar{u}a_\sigma^3$.  It follows that there is a $d_3$-differential 
$$d_3(u_{2\sigma}) = \bar{u}_1\bar{u} a_\sigma^3$$
in $C_2\text{-}\HFPSS(E_n)$.  Multiplying this differential by the permanent cycles produced before determines the rest of the $d_3$-differentials.  These are all the $d_3$-differentials because there are no more room for other $d_3$-differentials after these differentials. 

Suppose now that the induction hypothesis holds for all $1 \leq k \leq r-1 < n$.  For degree reasons, after the $d_{2^{r} -1}$-differentials, the next possible differential is of length $d_{2^{r+1}-1}$.  In $C_2\text{-}\HFPSS(BP_\mathbb{R})$, there is a $d_{2^{r+1}-1}$-differential 
$$d_{2^{r+1}-1}(u_{2\sigma}^{2^{r-1}}) = \bar{v}_{r}a_\sigma^{2^{r+1}-1}.$$
The map 
$$ C_2\text{-}\HFPSS(BP_\mathbb{R}) \to C_2\text{-}\HFPSS(E_n)$$
of spectral sequences sends the source to $u_{2\sigma}^{2^{r-1}}$ and the target to 
$$\bar{v}_{r}a_\sigma^{2^{r+1}-1} \mapsto \left\{
\begin{array}{ll} 
\bar{u}_{r}\bar{u}^{2^{r}-1}a_\sigma^{2^{r+1}-1}  & r < n \\ 
\bar{u}^{2^n-1}a_\sigma^{2^{n+1}-1} & r = n.
\end{array}
\right. $$
In particular, both images are not zero.  Moreover, the image of the target must be killed by a differential of length at most $2^{r+1}-1$.  By degree reasons, the image of the target cannot be killed by a shorter differential.  It follows that there is a $d_{2^{r+1}-1}$-differential 

$$d_{2^{r+1}-1}(u_{2\sigma}^{2^{r-1}}) = \left\{
\begin{array}{ll} 
\bar{u}_{r}\bar{u}^{2^{r}-1}a_\sigma^{2^{r+1}-1}  & r < n \\ 
\bar{u}^{2^n-1}a_\sigma^{2^{n+1}-1} & r = n.
\end{array}
\right.$$
The rest of the $d_{2^{r+1}-1}$-differentials are produced by multiplying this differential with permanent cycles.  After these differentials, there are no room for other $d_{2^{r+1}-1}$-differentials by degree reasons.  This concludes the proof of the theorem. 
\hfill $\square$ 

\begin{rmk}\rm
As an example, Figures~\ref{fig:C2HFPSSEnE3}--\ref{fig:C2HFPSSEnEinfty} show the differentials in the integer-graded part of $C_2\text{-}\HFPSS(E_3)$.  The spectral sequence converges after the $E_{15}$-page and we learn that $\pi_* E_3^{hC_2}$ is 32-periodic.  
\end{rmk}

\begin{figure}
\begin{center}
\makebox[\textwidth]{\hspace{-0.7in}\includegraphics[trim={0cm 7cm 1.5cm 2cm},clip,page = 1, scale = 1]{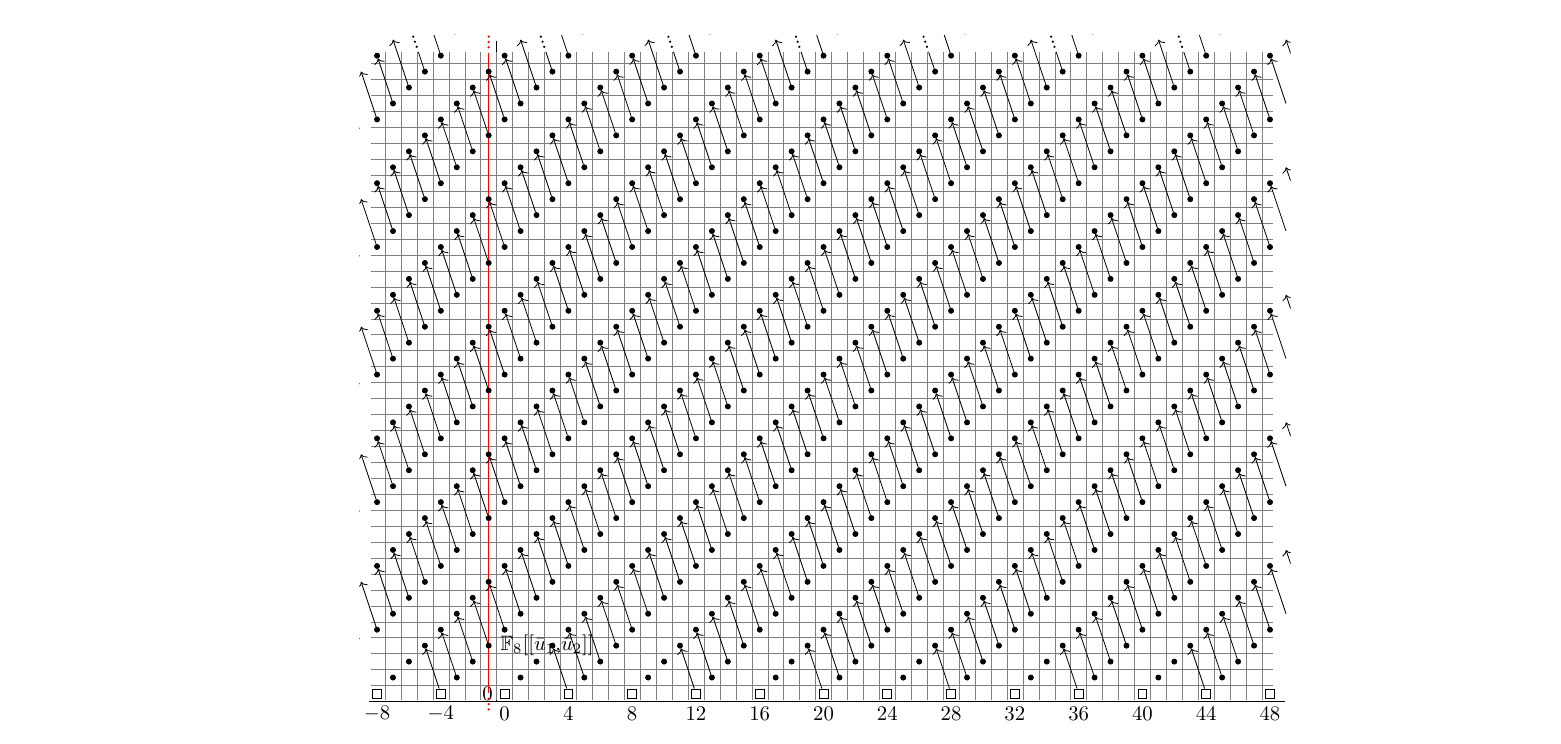}}
\end{center}
\begin{center}
\caption{$d_3$-differentials in the integer graded part of $C_2\text{-}\HFPSS(E_3)$.}
\label{fig:C2HFPSSEnE3}
\end{center}
\end{figure}

\begin{figure}
\begin{center}
\makebox[\textwidth]{\hspace{1in}\includegraphics[trim={1cm 7cm 0cm 2cm},clip,page = 2, scale = 1]{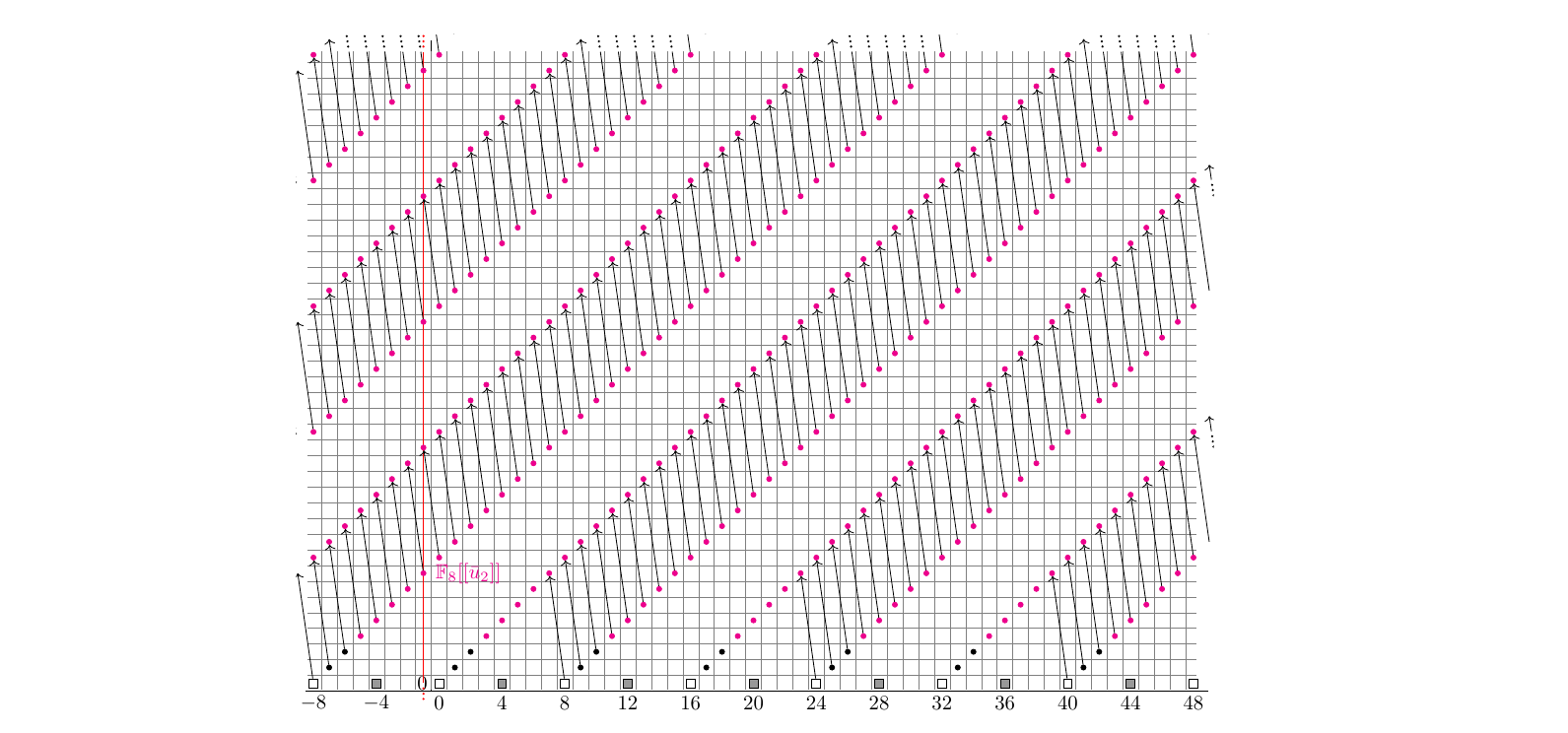}}
\end{center}
\begin{center}
\caption{$d_7$-differentials in the integer graded part of $C_2\text{-}\HFPSS(E_3)$.}
\label{fig:C2HFPSSEnE7}
\end{center}
\end{figure}

\begin{figure}
\begin{center}
\makebox[\textwidth]{\hspace{-0.7in}\includegraphics[trim={0cm 7cm 1.5cm 2cm},clip,page = 3, scale = 1]{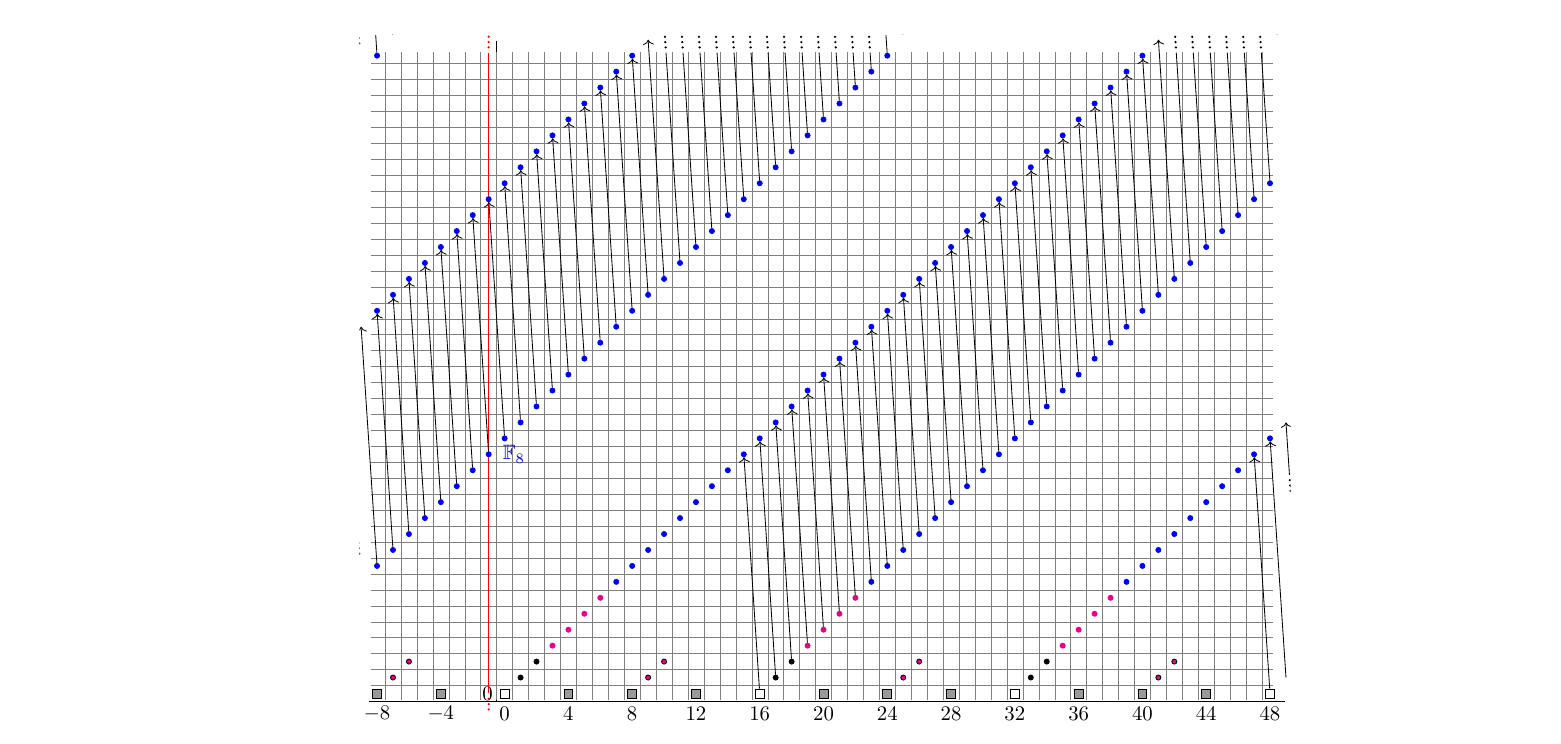}}
\end{center}
\begin{center}
\caption{$d_{15}$-differentials in the integer graded part of $C_2\text{-}\HFPSS(E_3)$.}
\label{fig:C2HFPSSEnE15}
\end{center}
\end{figure}

\begin{figure}
\begin{center}
\makebox[\textwidth]{\hspace{1in}\includegraphics[trim={1cm 7cm 0cm 2cm},clip,page = 4, scale = 1]{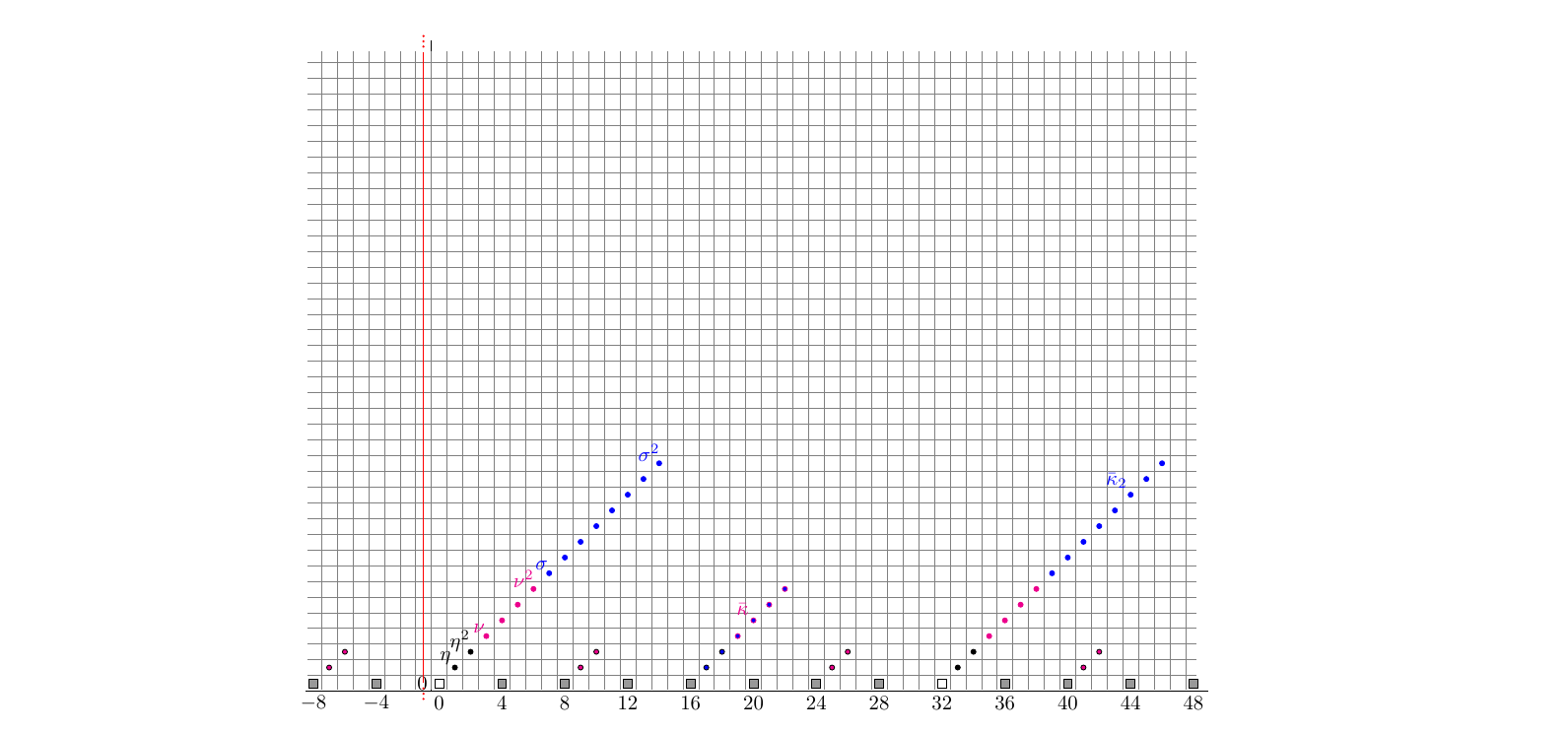}}
\end{center}
\begin{center}
\caption{$E_\infty$-page of the integer graded part of $C_2\text{-}\HFPSS(E_3)$.}
\hfill
\label{fig:C2HFPSSEnEinfty}
\end{center}
\end{figure}

\subsection{Real Landweber Exactness}
We will now use the $C_2$-homotopy fixed point spectral sequence of $E_n$ to show that $E_n$ is Real Landweber exact.  First, we will recall some definitions and theorems from \cite{HillMeier}. 

\begin{dfn}[\cite{Araki}]\rm
Let $E$ be a $C_2$-equivariant homotopy commutative ring spectrum.  A \textit{Real orientation} of $E$ is a class $\bar{x} \in \widetilde{E}_{C_2}^{\rho}(\mathbb{CP}^\infty)$ whose restriction to 
$$\widetilde{E}_{C_2}^{\rho} (\mathbb{CP}^1) = \widetilde{E}_{C_2}^{\rho}(S^{\rho}) \cong E^0_{C_2}(pt)$$
is the unit.  Here, we are viewing $\mathbb{CP}^n$ as a $C_2$-space via complex conjugation.  
\end{dfn}

By \cite[Theorem~2.25]{HuKriz}, Real orientations of $E$ are in one-to-one correspondence with homotopy commutative maps $MU_\mathbb{R} \to E$ of $C_2$-ring spectra.  

\begin{dfn} \rm (\cite[Definition~3.1]{HillMeier}). 
A $C_2$-spectrum $E\mathbb{R}$ is \textit{even} if $\underline{\pi}_{k\rho-1}E\mathbb{R} = 0$ for all $k \in \mathbb{Z}$.  It is called \textit{strongly even} if additionally $\underline{\pi}_{k\rho}E\mathbb{R}$ is a constant Mackey functor for all $k \in \mathbb{Z}$, i.e., if the restriction 
$$\pi_{k\rho}^{C_2}E\mathbb{R} \to \pi^e_{k\rho}E\mathbb{R} \cong \pi_{2k}^e E\mathbb{R}$$
is an isomorphism. 
\end{dfn}

Even spectra satisfy very nice properties.  In particular, Hill--Meier further proved (\cite[Lemma~3.3]{HillMeier}) that if a $C_2$-spectrum $E\mathbb{R}$ is even, then $E\mathbb{R}$ is Real orientable.  They proved this by showing that all the obstructions to having a Real orientation lie in the groups $\pi_{2k-1}E\mathbb{R}$ and $\pi_{k\rho-1}^{C_2}E\mathbb{R}$, which are all 0 by definition.  

\begin{dfn}\rm (\cite[Definition~3.5]{HillMeier}). 
Let $E\mathbb{R}$ be a strongly even $C_2$-spectrum with underlying spectrum $E$.  Then $E\mathbb{R}$ is called \textit{Real Landweber exact} if for every Real orientation $MU_\mathbb{R} \to E\mathbb{R}$ the induced map 
$${MU_\mathbb{R}}_\bigstar(X) \otimes_{MU_{2*}} E_{2*} \to E\mathbb{R}_{\bigstar}(X) $$
is an isomorphism for every $C_2$-spectrum $X$. 
\end{dfn}
Here, we are treating ${MU_\mathbb{R}}_\bigstar$ as a graded $MU_{2*}$-module because the restriction map $({MU_\mathbb{R}})_{k\rho} \to MU_{2k}$ is an isomorphism, and it defines a graded ring morphism $MU_{2*} \to {MU_\mathbb{R}}_\bigstar$ by sending elements of degree $2k$ to elements of degree $k\rho$.  

\begin{thm}[\cite{HillMeier}, Real Landweber exact functor theorem]\label{thm:HillMeierRealLEFT}
Let $E\mathbb{R}$ be a strongly even $C_2$-spectrum whose underlying spectrum $E$ is Landweber exact.  Then $E\mathbb{R}$ is Real Landweber exact. 
\end{thm}

For $E_n$, its underlying spectrum is clearly Landweber exact.  In light of Theorem~\ref{thm:HillMeierRealLEFT}, we prove the following: 

\begin{thm}
$E_n$ is a Real Landweber exact spectrum.  
\end{thm}
\begin{proof}
By Theorem~\ref{thm:HillMeierRealLEFT}, it suffices to show that $E_n$ is strongly even.  By Thereom~\ref{FullVersion2}, the classes $\bar{u}_1$, $\ldots$, $\bar{u}_{n-1}$, and $\bar{u}^\pm$ are permanent cycles in $C_2\text{-}\HFPSS(E_n)$.  The restriction of these classes to $\pi_{2*}^eE_n$ are $u_1$, $\ldots$, $u_{n-1}$, and $u^\pm$, respectively.  Furthermore, there are no other classes in $\pi_{*\rho}^{C_2} E_n$.  This shows that the restriction map 
$$\pi_{*\rho}^{C_2} E_n \to \pi_{2*}^e E_n$$
is an isomorphism, hence $\underline{\pi}_{k\rho} E_n$ is a constant Mackey functor for all $k \in \mathbb{Z}$.  

Classically, we already know that $\pi_{2k-1}^e E_n = 0$.  The following lemma shows that $\underline{\pi}_{k\rho -1}E_n = 0$ for all $k \in \mathbb{Z}$.
\end{proof}

\begin{lem}
The groups $\pi_{k\rho-1}^{C_2} E_n = 0$ for all $k \in \mathbb{Z}$.
\end{lem}
\begin{proof}
In $C_2\text{-}\HFPSS(E_n)$, the classes $\bar{u}^\pm$ are permanent cycles.  Since $|\bar{u}| = \rho$, multiplying by $\bar{u}^k$ produces an isomorphism
$$\pi_{\bigstar}^{C_2}E_n \stackrel{\cong}{\to} \pi_{\bigstar+k\rho}^{C_2} E_n.$$
It follows that in order to show $\pi_{k\rho-1}^{C_2}E_n = 0$ for $k \in \mathbb{Z}$, it suffices to prove $\pi_{-1}^{C_2}E_n = 0$.

Recall that the $E_2$-page of $C_2\text{-}\HFPSS(E_n)$ is 
$$E_2^{s, t} (E_n^{hC_2}) = W(\mathbb{F}_{2^n})[[\bar{u}_1, \ldots, \bar{u}_{n-1}]][\bar{u}^\pm] \otimes \mathbb{Z}[u_{2\sigma}^\pm, a_\sigma]/(2a_\sigma).$$
As in Figure~\ref{fig:C2HFPSSEnEinfty}, every class on the 0-line is of the form 
$$W(\mathbb{F}_{2^n})[[\bar{u}_1, \ldots, \bar{u}_{n-1}]]\bar{u}^au_{2\sigma}^b,$$
where $a, b \in \mathbb{Z}$, and every class of filtration greater than 0 is of the form 
$$\mathbb{F}_{2^n}[[\bar{u}_1, \ldots, \bar{u}_{n-1}]]\bar{u}^au_{2\sigma}^ba_\sigma^c,$$
where $a, b \in \mathbb{Z}$, and $c >0$.  For degree reasons, the classes on the $(-1)$-stem are all of the form 
$$\mathbb{F}_{2^n}[[\bar{u}_1, \ldots, \bar{u}_{n-1}]]\bar{u}^{2\ell -1} u_{2\sigma}^{-\ell}a_\sigma^{4\ell -1},$$
where $\ell \geq 1$.  The relevant differentials that have source or target in the $(-1)$-stem are all generated by
$$d_{2^{r+1}-1}(u_{2\sigma}^{-2^{r-1}}) = d_{2^{r+1}-1}(u_{2\sigma}^{-2^r}\cdot u_{2\sigma}^{2^{r-1}}) = u_{2\sigma}^{-2^r} \cdot d_{2^{r+1}-1}(u_{2\sigma}^{2^{r-1}}) = \left\{
\begin{array}{ll} 
\bar{u}_{r}\bar{u}^{2^{r}-1} u_{2\sigma}^{-2^r} a_\sigma^{2^{r+1}-1}  & 0< r < n \\ 
\bar{u}^{2^n-1}u_{2\sigma}^{-2^n}a_\sigma^{2^{n+1}-1} & r = n.
\end{array}
 \right. $$
We will analyze these differentials one-by-one: \\

\noindent (1) The relevant $d_3$-differentials are all generated by the differential 
$$d_3(u_{2\sigma}^{-1}) = \bar{u}_1\bar{u}u_{2\sigma}^{-2}a_\sigma^3.$$
The classes at $\bar{u}^{2\ell-1}u_{2\sigma}^{-\ell}a_\sigma^{4\ell -1}$, with $\ell \equiv 1 \pmod{2}$, are the sources of these differentials, and hence they die after the $E_3$-page.  The classes at $\bar{u}^{2\ell-1}u_{2\sigma}^{-\ell}a_\sigma^{4\ell -1}$, with $\ell \equiv 0 \pmod{2}$, are the targets.  These differentials quotient out the principal ideal $(\bar{u}_1)$ at these targets.  The remaining classes at these targets are of the form 
$$\mathbb{F}_{2^n}[[\bar{u}_2, \ldots, \bar{u}_{n-1}]]\bar{u}^{2\ell -1} u_{2\sigma}^{-\ell}a_\sigma^{4\ell -1},$$
with $\ell \equiv 0 \pmod{2}$.  \\

\noindent (2) The relevant $d_{7}$-differentials are all generated by the differential 
$$d_{7}(u_{2\sigma}^{-2}) = \bar{u}_{2}\bar{u}^{3} u_{2\sigma}^{-4} a_\sigma^{7}.$$
The classes at $\bar{u}^{2\ell-1}u_{2\sigma}^{-\ell}a_\sigma^{4\ell -1}$, with $\ell \equiv 2 \pmod{4}$, are the sources of these differentials, and hence they die after the $E_{7}$-page.  The classes at $\bar{u}^{2\ell-1}u_{2\sigma}^{-\ell}a_\sigma^{4\ell -1}$, with $\ell \equiv 0 \pmod{4}$, are the targets.  These differentials quotient out the principal ideal $(\bar{u}_2)$ at these targets.  The remaining classes at these targets are of the form 
$$\mathbb{F}_{2^n}[[\bar{u}_3, \ldots, \bar{u}_{n-1}]]\bar{u}^{2\ell -1} u_{2\sigma}^{-\ell}a_\sigma^{4\ell -1},$$
with $\ell \equiv 0 \pmod{4}$.  \\

\noindent (3) In general, for $0 < r < n$, the relevant $d_{2^{r+1}-1}$-differentials are all generated by the differential 
$$d_{2^{r+1}-1}(u_{2\sigma}^{-2^{r-1}}) = \bar{u}_{r}\bar{u}^{2^{r}-1} u_{2\sigma}^{-2^r} a_\sigma^{2^{r+1}-1}.$$
The classes at $\bar{u}^{2\ell-1}u_{2\sigma}^{-\ell}a_\sigma^{4\ell -1}$, with $\ell \equiv 2^{r-1} \pmod{2^r}$, are the sources of these differentials, and hence they die after the $E_{2^{r+1}-1}$-page.  The classes at $\bar{u}^{2\ell-1}u_{2\sigma}^{-\ell}a_\sigma^{4\ell -1}$, with $\ell \equiv 0 \pmod{2^r}$, are the targets.  These differentials quotient out the principal ideal $(\bar{u}_r)$ at these targets.  The remaining classes at these targets are of the form 
$$\mathbb{F}_{2^n}[[\bar{u}_{r+1}, \ldots, \bar{u}_{n-1}]]\bar{u}^{2\ell -1} u_{2\sigma}^{-\ell}a_\sigma^{4\ell -1},$$
with $\ell \equiv 0 \pmod{2^r}$.  \\

\noindent (4) The relevant $d_{2^{n+1}-1}$-differentials are all generated by the differential 
$$d_{2^{n+1}-1}(u_{2\sigma}^{-2^{n-1}}) = \bar{u}^{2^n-1}u_{2\sigma}^{-2^n}a_\sigma^{2^{n+1}-1}.$$
The classes at $\bar{u}^{2\ell-1}u_{2\sigma}^{-\ell}a_\sigma^{4\ell -1}$, with $\ell \equiv 2^{n-1} \pmod{2^n}$, are the sources of these differentials, and hence they die after the $E_{2^{n+1}-1}$-page.  The classes at $\bar{u}^{2\ell-1}u_{2\sigma}^{-\ell}a_\sigma^{4\ell -1}$, with $\ell \equiv 0 \pmod{2^n}$, are the targets.  They also die after these differentials because the only classes at these targets now are $\bar{u}^{2\ell -1} u_{2\sigma}^{-\ell}a_\sigma^{4\ell -1}$.\\

It follows that every class at the $(-1)$-stem vanish after the $E_{2^{n+1}-1}$-page.  This implies $\pi_{-1}^{C_2} E_n = 0$, as desired.  

\end{proof}

\section{Hurewicz Images}\label{sec:HurewiczImages}

In this section, we will prove that $\pi_* E_n^{hC_2}$ detects the Hopf elements, the Kervaire classes, and the $\bar{\kappa}$-family.  The case when $n=1$ and $n=2$ are previously known.  When $n =1$, $E_1 = KU^{\wedge}_2$ and $E_1^{hC_2} = KO^{\wedge}_2$.  It is well-known that $\pi_*KO^{\wedge}_2$ detects $\eta \in \pi_1 \mathbb{S}$ and $\eta^2 \in \pi_2 \mathbb{S}$ (\cite{AtiyahKR}). When $n = 2$, the Mahowald--Rezk transfer argument (\cite{MahowaldRezk}) shows that $\pi_*E_2^{hC_2}$ detects $\eta$, $\eta^2$, $\nu \in \pi_3 \mathbb{S}$, $\nu^2 \in \pi_6 \mathbb{S}$, and $\bar{\kappa} \in \pi_{20}\mathbb{S}$.

\vspace{\baselineskip}


The Hopf elements are represented by the elements 
$$h_i \in \text{Ext}_{\mathcal{A}_*}^{1, 2^i}(\mathbb{F}_2, \mathbb{F}_2)$$ 
on the $E_2$-page of the classical Adams spectral sequence at the prime 2.  By Adam's solution of the Hopf invariant one problem \cite{AdamsHopfInvariant}, only $h_0$, $h_1$, $h_2$, and $h_3$ survive to the $E_\infty$-page.  By Browder's work \cite{Browder}, the Kervaire classes $\theta_j \in \pi_{2^{j+1}-2}\mathbb{S}$, if they exist, are represented by the elements 
$$h_j^2 \in \text{Ext}_{\mathcal{A}_*}^{2, 2^{j+1}}(\mathbb{F}_2, \mathbb{F}_2)$$
on the $E_2$-page.  For $j \leq 5$, $h_j^2$ survive.  The case $\theta_4 \in \pi_{30} \mathbb{S}$ is due to Barratt--Mahowald--Tangora \cite{MahowaldTangoraTheta4, BarratMahowaldTangoraTheta4}, and the case $\theta_5 \in \pi_{62} \mathbb{S}$ is due to Barratt--Jones--Mahowald \cite{BarrattJonesMahowaldTheta5}. The fate of $h_6^2$ is unknown.  Hill--Hopkins--Ravenel's result \cite{HHR} shows that the $h_j^2$, for $j \geq 7$, do not survive to the $E_\infty$-page.  

To introduce the $\bar{\kappa}$-family, we appeal to Lin's complete classification of $\Ext^{\leq 4, t}_{\mathcal{A}_*}(\F, \F)$ in \cite{LinExt}.  In his classification, Lin showed that there is a family $\{g_k \,| \, k \geq 1\}$ of indecomposable elements with 
$$g_k \in \Ext^{4, 2^{k+2} + 2^{k+3}}_{\mathcal{A}_*}(\F, \F).$$
The first element of this family, $g_1$, is in bidegree $(4, 24)$.  It survives the Adams spectral sequence to become $\bar{\kappa} \in \pi_{20} \mathbb{S}$.  It is for this reason that we name this family the $\bar{\kappa}$-family.  The element $g_2$ also survives to become the element $\bar{\kappa}_2 \in \pi_{44}\mathbb{S}$.  For $k \geq 3$, the fate of $g_k$ is unknown.  \\

In \cite{HurewiczImages}, the second author, together with Li, Wang, and Xu, proved detection theorems for the Hurewicz images of $MU_\mathbb{R}^{C_2} \approx MU_\mathbb{R}^{hC_2}$ and $BP_\mathbb{R}^{C_2} \approx BP_\mathbb{R}^{hC_2}$ (the equivalences between the $C_2$-fixed points and the $C_2$-homotopy fixed points for $MU_\mathbb{R}$ and $BP_\mathbb{R}$ are due to Hu and Kriz \cite[Theorem~4.1]{HuKriz}).  

\begin{thm}\label{thm:MURBPRDetection} (Li--Shi--Wang--Xu, Detection Theorems for $MU_\mathbb{R}$ and $BP_\mathbb{R}$). 
The Hopf elements, the Kervaire classes, and the $\bar{\kappa}$-family are detected by the Hurewicz maps $\pi_*\mathbb{S} \to \pi_* MU_\mathbb{R}^{C_2}\cong \pi_* MU_\mathbb{R}^{hC_2}$ and $\pi_*\mathbb{S} \to \pi_* BP_\mathbb{R}^{C_2} \cong \pi_* BP_\mathbb{R}^{hC_2}$.  
\end{thm}

Given the discussion above, Theorem~\ref{thm:MURBPRDetection} shows that the elements $\eta$, $\nu$, $\sigma$, and $\theta_j$, for $1 \leq j \leq 5$, are detected by $\pi_*^{C_2} MU_\mathbb{R}^{hC_2}$ and $\pi_*^{C_2} BP_\mathbb{R}^{hC_2}$. The last unknown Kervaire class $\theta_6$ and the classes $g_k$ for $k \geq 3$ will also be detected, should they survive the Adams spectral sequence. 

The proof of Theorem~\ref{thm:MURBPRDetection} requires the $C_2$-equivariant Adams spectral sequence developed by Greenlees \cite{GreenleesThesis,Greenlees1988, Greenlees1990} and Hu--Kriz \cite{HuKriz}.  Since $MU_\mathbb{R}$ splits as a wedge of suspensions of $BP_\mathbb{R}$ 2-locally, we only need to focus on $BP_\mathbb{R}$.  There is a map of Adams spectral sequences 
$$\begin{tikzcd} 
\text{classical Adams spectral sequence of } \mathbb{S} \ar[r, Rightarrow] \ar[d] & (\pi_*\mathbb{S})^\wedge_2 \ar[d] \\ 
C_2\text{-equivariant Adams spectral sequence of }\mathbb{S} \ar[r, Rightarrow] \ar[d] & (\pi_\bigstar^{C_2} F(E{C_2}_+, \mathbb{S}))^\wedge_2 \ar[d] \\ 
C_2\text{-equivariant Adams spectral sequence of } {BP_\mathbb{R}} \ar[r, Rightarrow] & (\pi_\bigstar^{C_2} F(E{C_2}_+,BP_\mathbb{R}))^\wedge_2.
\end{tikzcd}$$

It turns out that for degree reasons, the $C_2$-equivariant Adams spectral sequence for $BP_\mathbb{R}$ degenerates at the $E_2$-page.  From this, Theorem~\ref{thm:MURBPRDetection} follows easily from the following algebraic statement:
\begin{thm}[Li--Shi--Wang--Xu, Algebraic Detection Theorem]\label{algebraicDetection}
The images of the elements $\{h_i \,|\, i \geq 1\}$, $\{h_j^2 \,|\, j \geq 1\}$, and $\{g_k \,|\, k \geq 1\}$ on the $E_2$-page of the classical Adams spectral sequence of $\mathbb{S}$ are nonzero on the $E_2$-page of the $C_2$-equivariant Adams spectral sequence of $BP_\mathbb{R}$.  
\end{thm}

The proof of Theorem~\ref{algebraicDetection} requires an analysis of the algebraic maps 
$$\text{Ext}_{\mathcal{A}_*}(\mathbb{F}_2, \mathbb{F}_2) \to \text{Ext}_{\mathcal{A}_\bigstar^{cc}}(H^c_\bigstar, H^c_\bigstar) \to \text{Ext}_{\Lambda_\bigstar^{cc}}(H^c_\bigstar, H^c_\bigstar).$$
These are the maps on the $E_2$-pages of the Adams spectral sequences above.  Here, $\mathcal{A}_*:= (H\mathbb{F}_2 \wedge H\mathbb{F}_2)_*$ is the classical dual Steenrod algebra; $H^c_\bigstar := F(E{C_2}_+, H\mathbb{F}_2)_\bigstar$ is the Borel $C_2$-equivariant Eilenberg--MacLane spectrum; $\mathcal{A}_\bigstar^{cc}:= F(E{C_2}_+, H\mathbb{F}_2 \wedge H\mathbb{F}_2)_\bigstar$ is the Borel $C_2$-equivariant dual Steenrod algebra; and $\Lambda_\bigstar^{cc}$ is a quotient of $\mathcal{A}_\bigstar^{cc}$.  Hu and Kriz \cite{HuKriz} studied $\mathcal{A}_\bigstar^{cc}$ and completely computed the Hopf algebroid structure of $(H^c_\bigstar, \mathcal{A}_\bigstar^{cc})$.  Using their formulas, it is possible to compute the map 
$$(H\mathbb{F}_2, \mathcal{A}_*) \to (H^c_\bigstar, \mathcal{A}_\bigstar^{cc}) \to (H^c_\bigstar, \Lambda_\bigstar^{cc})$$
of Hopf-algebroids.  Filtering these Hopf algebroids compatibly produces maps of May spectral sequences: 
$$\begin{tikzcd}
\text{May spectral sequence of } \mathbb{S} \ar[r, Rightarrow] \ar[d] & \text{Ext}_{\mathcal{A}_*}(\mathbb{F}_2, \mathbb{F}_2) \ar[d] \\ 
C_2\text{-equivariant May spectral sequence of }\mathbb{S} \ar[r, Rightarrow] \ar[d] & \text{Ext}_{\mathcal{A}_\bigstar^{cc}}(H^c_\bigstar, H^c_\bigstar) \ar[d] \\ 
C_2\text{-equivariant May spectral sequence of } {BP_\mathbb{R}} \ar[r, Rightarrow] &  \text{Ext}_{\Lambda_\bigstar^{cc}}(H^c_\bigstar, H^c_\bigstar).
\end{tikzcd}$$

There is a connection between the $C_2$-equivariant May spectral sequence of $BP_\mathbb{R}$ and the homotopy fixed point spectral sequence of $BP_\mathbb{R}$: 

\begin{thm} [Li--Shi--Wang--Xu] \label{thm:AccHFPSS}
The $C_2$-equivariant May spectral sequence of $BP_\mathbb{R}$ is isomorphic to the associated-graded homotopy fixed point spectral sequence of $BP_\mathbb{R}$ as $RO(C_2)$-graded spectral sequences.  
\end{thm}

By the ``associated-graded homotopy fixed point spectral sequence'', we mean that whenever we see a $\mathbb{Z}$-class on the $E_2$-page, we replace it by a tower of $\mathbb{Z}/2$-classes.  Since the equivariant Adams spectral sequence of $BP_\mathbb{R}$ degenerates, the $E_\infty$-page of the $C_2$-equivariant May spectral sequence of $BP_\mathbb{R}$ is an associated-graded of $\pi_\bigstar^{C_2} F(E{C_2}_+, BP_\mathbb{R})$.  The isomorphism in Theorem~\ref{thm:AccHFPSS} allows us to identify the classes in $C_2\text{-}\HFPSS(E_n)$ that detects the Hopf elements, the Kervaire classes, and the $\bar{\kappa}$-family.  This is crucial for tackling detection theorems of $E_n^{hC_2}$.  

Using Hu--Kriz's formulas, one can compute the maps on the $E_2$-pages of the May spectral sequences above, as well as all the differentials in the $C_2$-equivariant May spectral sequence of $BP_\mathbb{R}$.  

\begin{thm} [Li--Shi--Wang--Xu] \label{thm:E2MayImage}
On the $E_2$-page of the map 
$$\MaySS(\mathbb{S}) \to \EMaySS(\mathbb{S}) \to \EMaySS(BP_\mathbb{R}) \cong C_2\text{-}\HFPSS(BP_\mathbb{R}),$$
The classes 
\begin{eqnarray}
h_i &\mapsto& \bar{v}_ia_\sigma^{2^i -1},  \nonumber \\ 
h_j^2 &\mapsto& \bar{v}_j^2a_\sigma^{2(2^j -1)},  \nonumber \\ 
h_{2k}^4 &\mapsto& \bar{v}_{k+1}^4u_{2\sigma}^{2^{k+1}}a_\sigma^{4(2^k-1)}. \nonumber 
\end{eqnarray}
These classes all survive to the $E_\infty$-page in $C_2\text{-}\HFPSS(BP_\mathbb{R})$.  
\end{thm}

Since $E_n$ is Real oriented and everything is 2-local, a Real orientation gives us a $C_2$-equivariant homotopy commutative map 
$$BP_\mathbb{R} \to E_n,$$ 
which induces a multiplicative map 
$$C_2\text{-}\HFPSS(BP_\mathbb{R}) \to C_2\text{-}\HFPSS(E_n)$$
of spectral sequences.  On the $E_2$-page, this map sends the classes $u_{2\sigma} \mapsto u_{2\sigma}$, $\as \mapsto \as$, and 
\begin{eqnarray} \label{eqn:BPREn}
\bar{v}_i \mapsto \left \{ \begin{array}{ll} \bar{u}_i \bar{u}^{2^i -1} & 1 \leq i \leq n-1 \\ 
\bar{u}^{2^n-1} & i = n \\ 
0 & i > n.
\end{array} \right.
\end{eqnarray}

\begin{thm} [Detection Theorem for $E_n^{hC_2}$] \hfill
\label{EnC2Detection}
\begin{enumerate}
\item For $1 \leq i, j \leq n$, if the element $h_{i} \in \Ext_{\mathcal{A}_*}^{1, 2^i}(\F, \F)$ or $h_{j}^2 \in \Ext_{\mathcal{A}_*}^{2, 2^{j+1}}(\F, \F)$ survives to the $E_\infty$-page of the Adams spectral sequence, then its image under the Hurewicz map $\pi_*\mathbb{S} \to \pi_* E_n^{hC_2}$ is nonzero.  
\item For $1 \leq k \leq n-1$, if the element $g_k \in \Ext_{\mathcal{A}_*}^{4, 2^{k+2}+2^{k+3}}(\F, \F)$ survives to the $E_\infty$-page of the Adams spectral sequence, then its image under the Hurewicz map $\pi_*\mathbb{S} \to \pi_*E_n^{hC_2}$ is nonzero. 
\end{enumerate}
\end{thm}

\begin{proof}
By Theorem~\ref{thm:E2MayImage} and (\ref{eqn:BPREn}), the composite map
$$\MaySS(\mathbb{S}) \to \EMaySS(BP_\mathbb{R}) \cong C_2\text{-}\HFPSS(BP_\mathbb{R}) \to C_2\text{-}\HFPSS(E_n)$$
on the $E_2$-pages sends the classes 
\begin{eqnarray}
h_i \mapsto \bar{v}_ia_\sigma^{2^i -1} 
&\mapsto& \left\{\begin{array}{ll}
\bar{u}_i \bar{u}^{2^i -1}a_\sigma^{2^i -1} & 1 \leq i \leq n-1 \\ 
\bar{u}^{2^n-1}a_\sigma^{2^n -1} & i = n \\ 
0 & i > n,
\end{array} \right.  \nonumber \\ 
h_j^2 \mapsto \bar{v}_j^2a_\sigma^{2(2^j -1)} 
&\mapsto& \left\{\begin{array}{ll}
\bar{u}_j^2 \bar{u}^{2(2^j -1)}a_\sigma^{2(2^j -1)} & 1 \leq j \leq n-1 \\ 
\bar{u}^{2(2^n-1)}a_\sigma^{2(2^n -1)} & j = n \\ 
0 & j > n,
\end{array} \right.  \nonumber \\ 
h_{2k}^4 \mapsto \bar{v}_{k+1}^4u_{2\sigma}^{2^{k+1}}a_\sigma^{4(2^k-1)}
&\mapsto& \left\{\begin{array}{ll}
\bar{u}_{k+1}^4\bar{u}^{4(2^{k+1}-1)}u_{2\sigma}^{2^{k+1}} a_\sigma^{4(2^k-1)} & 1 \leq k \leq n-2 \\ 
\bar{u}^{4(2^n-1)}u_{2\sigma}^{2^n}a_\sigma^{4(2^{n-1}-1)} & k = n-1 \\ 
0 & k > n-1.
\end{array} \right.  \nonumber 
\end{eqnarray}
We know all the differentials in $C_2\text{-}\HFPSS(E_n)$ from Section~\ref{sec:HFPSSEn}.  From these differentials, it is clear that all the nonzero images on the $E_2$-page survive to the $E_\infty$-page to represent elements in $\pi_* E_n^{hC_2}$.  The statement of the theorem follows. 
\end{proof}

\begin{cor}[Detection Theorem for $E_n^{hG}$]\label{EnhGDetection}
Let $G$ be a finite subgroup of the Morava stabilizer group $\mathbb{G}_n$ containing the centralizer subgroup $C_2$.   
\begin{enumerate}
\item For $1 \leq i, j \leq n$, if the element $h_{i} \in \Ext_{\mathcal{A}_*}^{1, 2^i}(\F, \F)$ or $h_{j}^2 \in \Ext_{\mathcal{A}_*}^{2, 2^{j+1}}(\F, \F)$ survives to the $E_\infty$-page of the Adams spectral sequence, then its image under the Hurewicz map $\pi_*\mathbb{S} \to \pi_* E_n^{hG}$ is nonzero.  
\item For $1 \leq k \leq n-1$, if the element $g_k \in \Ext_{\mathcal{A}_*}^{4, 2^{k+2}+2^{k+3}}(\F, \F)$ survives to the $E_\infty$-page of the Adams spectral sequence, then its image under the Hurewicz map $\pi_*\mathbb{S} \to \pi_*E_n^{hG}$ is nonzero. 
\end{enumerate}
\end{cor}
\begin{proof}
Consider the following factorization of the unit map $\mathbb{S} \to E_n^{hC_2}$: 
$$\begin{tikzcd}
E_n^{hG} = F(EG_+, E_n)^G \ar[r] & F(EG_+, E_n)^{C_2} = E_n^{hC_2} \\ 
\mathbb{S} \ar[u] \ar[ru] &
\end{tikzcd}$$
The claims now follow easily from Theorem~\ref{EnC2Detection}.
\end{proof}

\bibliographystyle{alpha}
\bibliography{Bibliography}

\newcommand{\etalchar}[1]{$^{#1}$}
\begin{thebibliography}{GHMR05}

\bibitem[AA66]{AdamsAtiyahHopfInvariant}
J.~F. Adams and M.~F. Atiyah.
\newblock {$K$}-theory and the {H}opf invariant.
\newblock {\em Quart. J. Math. Oxford Ser. (2)}, 17:31--38, 1966.

\bibitem[ABG{\etalchar{+}}14]{ThomInfinity}
Matthew Ando, Andrew~J. Blumberg, David Gepner, Michael~J. Hopkins, and Charles
  Rezk.
\newblock An {$\infty$}-categorical approach to {$R$}-line bundles,
  {$R$}-module {T}hom spectra, and twisted {$R$}-homology.
\newblock {\em J. Topol.}, 7(3):869--893, 2014.

\bibitem[ACB19]{ThomRing}
Omar Antol\'{\i}n-Camarena and Tobias Barthel.
\newblock A simple universal property of {T}hom ring spectra.
\newblock {\em J. Topol.}, 12(1):56--78, 2019.

\bibitem[Ada60]{AdamsHopfInvariant}
J.~F. Adams.
\newblock On the non-existence of elements of {H}opf invariant one.
\newblock {\em Ann. of Math. (2)}, 72:20--104, 1960.

\bibitem[Ada62]{AdamsVectorFields}
J.~F. Adams.
\newblock Vector fields on spheres.
\newblock {\em Ann. of Math. (2)}, 75:603--632, 1962.

\bibitem[Ang08]{AngeltveitAinfty}
Vigleik Angeltveit.
\newblock Topological {H}ochschild homology and cohomology of {$A_\infty$} ring
  spectra.
\newblock {\em Geom. Topol.}, 12(2):987--1032, 2008.

\bibitem[Ara79]{Araki}
Sh\^or\^o Araki.
\newblock Orientations in {$\tau $}-cohomology theories.
\newblock {\em Japan. J. Math. (N.S.)}, 5(2):403--430, 1979.

\bibitem[Ati66]{AtiyahKR}
M.~F. Atiyah.
\newblock {$K$}-theory and reality.
\newblock {\em Quart. J. Math. Oxford Ser. (2)}, 17:367--386, 1966.

\bibitem[Bau08]{BauerTMF}
Tilman Bauer.
\newblock Computation of the homotopy of the spectrum {\tt tmf}.
\newblock In {\em Groups, homotopy and configuration spaces}, volume~13 of {\em
  Geom. Topol. Monogr.}, pages 11--40. Geom. Topol. Publ., Coventry, 2008.

\bibitem[BG18]{BobkovaGoerss}
I.~Bobkova and P.~G. Goerss.
\newblock Topological resolutions in {$K(2)$}-local homotopy theory at the
  prime {$2$}.
\newblock {\em Journal of Topology}, 11(4):917--956, 2018.

\bibitem[BGH18]{BeaudryGoerssHenn}
Agn\'es Beaudry, Paul~G. Goerss, and Hans-Werner Henn.
\newblock Chromatic splitting for the $k(2)$-local sphere at $p=2$.
\newblock {\em Arxiv 1712.08182}, 2018.

\bibitem[BJM84]{BarrattJonesMahowaldTheta5}
M.~G. Barratt, J.~D.~S. Jones, and M.~E. Mahowald.
\newblock Relations amongst {T}oda brackets and the {K}ervaire invariant in
  dimension {$62$}.
\newblock {\em J. London Math. Soc. (2)}, 30(3):533--550, 1984.

\bibitem[BMT70]{BarratMahowaldTangoraTheta4}
M.~G. Barratt, M.~E. Mahowald, and M.~C. Tangora.
\newblock Some differentials in the {A}dams spectral sequence. {II}.
\newblock {\em Topology}, 9:309--316, 1970.

\bibitem[BO16]{BehrensOrmsby}
Mark Behrens and Kyle Ormsby.
\newblock On the homotopy of {$Q(3)$} and {$Q(5)$} at the prime 2.
\newblock {\em Algebr. Geom. Topol.}, 16(5):2459--2534, 2016.

\bibitem[Bro69]{Browder}
William Browder.
\newblock The {K}ervaire invariant of framed manifolds and its generalization.
\newblock {\em Ann. of Math. (2)}, 90:157--186, 1969.

\bibitem[BSS20]{ThomQuotient}
Samik Basu, Steffen Sagave, and Christian Schlichtkrull.
\newblock {G}eneralized {T}hom spectra and their topological {H}ochschild
  homology.
\newblock {\em J. Inst. Math. Jussieu}, 19(1):21--64, 2020.

\bibitem[BW18]{BehrensWilson}
Mark Behrens and Dylan Wilson.
\newblock A {$C_2$}-equivariant analog of {M}ahowald's {T}hom spectrum theorem.
\newblock {\em Proc. Amer. Math. Soc.}, 146(11):5003--5012, 2018.

\bibitem[CM15]{ChadwickMandell}
Steven~Greg Chadwick and Michael~A. Mandell.
\newblock {$E_n$} genera.
\newblock {\em Geom. Topol.}, 19(6):3193--3232, 2015.

\bibitem[DH04]{DevinatzHopkins}
Ethan~S. Devinatz and Michael~J. Hopkins.
\newblock Homotopy fixed point spectra for closed subgroups of the {M}orava
  stabilizer groups.
\newblock {\em Topology}, 43(1):1--47, 2004.

\bibitem[DMPR17]{DottoRealTHH}
Emanuele Dotto, Kristian Moi, Irakli Patchkoria, and Sune~Precht Reeh.
\newblock Real topological hochschild homology.
\newblock {\em ar{X}iv preprint ar{X}iv:1711.10226}, 2017.

\bibitem[Dug05]{DuggerKR}
Daniel Dugger.
\newblock An {A}tiyah--{H}irzebruch spectral sequence for {KR}-theory.
\newblock {\em K-theory}, 35(3):213--256, 2005.

\bibitem[Fuj76]{FujiiMUR}
Michikazu Fujii.
\newblock Cobordism theory with reality.
\newblock {\em Math. J. Okayama Univ.}, 18(2):171--188, 1975/76.

\bibitem[GH04]{GoerssHopkins}
P.~G. Goerss and M.~J. Hopkins.
\newblock Moduli spaces of commutative ring spectra.
\newblock In {\em Structured ring spectra}, volume 315 of {\em London Math.
  Soc. Lecture Note Ser.}, pages 151--200. Cambridge Univ. Press, Cambridge,
  2004.

\bibitem[GHMR05]{GoerssHennMahowaldRezk}
P.~Goerss, H.-W. Henn, M.~Mahowald, and C.~Rezk.
\newblock A resolution of the {$K(2)$}-local sphere at the prime 3.
\newblock {\em Ann. of Math. (2)}, 162(2):777--822, 2005.

\bibitem[Gre85]{GreenleesThesis}
J.~P.~C. Greenlees.
\newblock Adams spectral sequences in equivariant topology.
\newblock {\em Ph.D. Thesis, University of Cambridge}, 1985.

\bibitem[Gre88]{Greenlees1988}
J.~P.~C. Greenlees.
\newblock Stable maps into free {$G$}-spaces.
\newblock {\em Trans. Amer. Math. Soc.}, 310(1):199--215, 1988.

\bibitem[Gre90]{Greenlees1990}
J.~P.~C. Greenlees.
\newblock The power of mod {$p$} {B}orel homology.
\newblock In {\em Homotopy theory and related topics ({K}inosaki, 1988)},
  volume 1418 of {\em Lecture Notes in Math.}, pages 140--151. Springer,
  Berlin, 1990.

\bibitem[Hen07]{Henn2007}
Hans-Werner Henn.
\newblock On finite resolutions of {$K(n)$}-local spheres.
\newblock In {\em Elliptic cohomology}, volume 342 of {\em London Math. Soc.
  Lecture Note Ser.}, pages 122--169. Cambridge Univ. Press, Cambridge, 2007.

\bibitem[HHR10]{HHRCDM1}
Michael~A Hill, Michael~J Hopkins, and Douglas~C Ravenel.
\newblock The arf-kervaire invariant problem in algebraic topology:
  introduction.
\newblock {\em Current developments in mathematics}, 2009:23--57, 2010.

\bibitem[HHR11]{HHRCDM2}
Michael~A Hill, Michael~J Hopkins, and Douglas~C Ravenel.
\newblock The arf-kervaire problem in algebraic topology: Sketch of the proof.
\newblock {\em Current developments in mathematics}, 2010:1--44, 2011.

\bibitem[HHR16]{HHR}
M.~A. Hill, M.~J. Hopkins, and D.~C. Ravenel.
\newblock On the nonexistence of elements of {K}ervaire invariant one.
\newblock {\em Ann. of Math. (2)}, 184(1):1--262, 2016.

\bibitem[HHR17]{HHRKH}
Michael~A. Hill, Michael~J. Hopkins, and Douglas~C. Ravenel.
\newblock The slice spectral sequence for the {$C_4$} analog of real
  {$K$}-theory.
\newblock {\em Forum Math.}, 29(2):383--447, 2017.

\bibitem[Hil17]{HillEquivDisc}
Michael~A. Hill.
\newblock On the algebras over equivariant little disks.
\newblock {\em ar{X}iv preprint ar{X}iv:1709.02005}, 2017.

\bibitem[HK01]{HuKriz}
Po~Hu and Igor Kriz.
\newblock Real-oriented homotopy theory and an analogue of the
  {A}dams-{N}ovikov spectral sequence.
\newblock {\em Topology}, 40(2):317 -- 399, 2001.

\bibitem[HM98]{HopkinsMahowald}
Michael~J. Hopkins and Mark Mahowald.
\newblock From elliptic curves to homotopy theory.
\newblock {\em \newline 1998 Preprint, {A}vailable at
  http://hopf.math.purdue.edu/Hopkins-Mahowald/eo2homotopy.pdf}, 1998.

\bibitem[HM17]{HillMeier}
Michael~A. Hill and Lennart Meier.
\newblock The {$C_2$}-spectrum {${\rm Tmf}_1(3)$} and its invertible modules.
\newblock {\em Algebr. Geom. Topol.}, 17(4):1953--2011, 2017.

\bibitem[Hop99]{COCTALOS}
Michael Hopkins.
\newblock Complex oriented cohomology theories and the language of stacks.
\newblock {\em \newline {A}vailable at
  http://web.math.rochester.edu/people/faculty/doug/otherpapers/coctalos.pdf},
  1999.

\bibitem[HSWX18]{HillShiWangXu}
Michael~A. Hill, XiaoLin~Danny Shi, Guozhen Wang, and Zhouli Xu.
\newblock The slice spectral sequence of a ${C}_4$-equivariant height-4
  {L}ubin--{T}ate theory.
\newblock {\em ArXiv 1811.07960}, 2018.

\bibitem[KLW18]{PhantomRing}
Nitu Kitchloo, Vitaly Lorman, and W.~Stephen Wilson.
\newblock Multiplicative structure on real {J}ohnson-{W}ilson theory.
\newblock In {\em New directions in homotopy theory}, volume 707 of {\em
  Contemp. Math.}, pages 31--44. Amer. Math. Soc., Providence, RI, 2018.

\bibitem[Lan68]{LandweberMUR}
Peter~S. Landweber.
\newblock Conjugations on complex manifolds and equivariant homotopy of {$MU$}.
\newblock {\em Bull. Amer. Math. Soc.}, 74:271--274, 1968.

\bibitem[Lin08]{LinExt}
Wen-Hsiung Lin.
\newblock $\text{Ext}_{A}^{4,\text{*}}(\mathbb{Z}/2, \mathbb{Z}/2)$ and
  $\text{Ext}_{A}^{5, \text{*}}(\mathbb{Z}/2, \mathbb{Z}/2)$.
\newblock {\em Topology and its Applications}, 155(5):459--496, 2008.

\bibitem[LSWX19]{HurewiczImages}
Guchuan Li, XiaoLin~Danny Shi, Guozhen Wang, and Zhouli Xu.
\newblock Hurewicz images of real bordism theory and real {J}ohnson-{W}ilson
  theories.
\newblock {\em Adv. Math.}, 342:67--115, 2019.

\bibitem[LT66]{LubinTate}
Jonathan Lubin and John Tate.
\newblock Formal moduli for one-parameter formal {L}ie groups.
\newblock {\em Bull. Soc. Math. France}, 94:49--59, 1966.

\bibitem[Lur10]{LurieChromatic}
Jacob Lurie.
\newblock Chromatic homotopy theory.
\newblock {\em \newline {A}vailable at
  http://www.math.harvard.edu/~lurie/252x.html}, 2010.

\bibitem[Lur18]{HA}
Jacob Lurie.
\newblock Higher {A}lgebra.
\newblock {\em \newline {A}vailable at http://www.math.harvard.edu/~lurie/},
  2018.

\bibitem[Mei18]{MeierTMF}
Lennart Meier.
\newblock Topological modular forms with level structure: decompositions and
  duality.
\newblock {\em ar{X}iv preprint ar{X}iv:1806.06709}, 2018.

\bibitem[Mil60]{Milnor60}
J.~Milnor.
\newblock On the cobordism ring {$\Omega \sp{\ast} $} and a complex analogue.
  {I}.
\newblock {\em Amer. J. Math.}, 82:505--521, 1960.

\bibitem[{Mil}11]{HaynesKervaire}
H.~{Miller}.
\newblock {Kervaire Invariant One [after M. A. Hill, M. J. Hopkins, and D. C.
  Ravenel]}.
\newblock {\em ArXiv e-prints}, April 2011.

\bibitem[MR09]{MahowaldRezk}
Mark Mahowald and Charles Rezk.
\newblock Topological modular forms of level 3.
\newblock {\em Pure Appl. Math. Q.}, 5(2, Special Issue: In honor of Friedrich
  Hirzebruch. Part 1):853--872, 2009.

\bibitem[MT67]{MahowaldTangoraTheta4}
Mark Mahowald and Martin Tangora.
\newblock Some differentials in the {A}dams spectral sequence.
\newblock {\em Topology}, 6:349--369, 1967.

\bibitem[Nov60]{Novikov60}
S.~P. Novikov.
\newblock Some problems in the topology of manifolds connected with the theory
  of {T}hom spaces.
\newblock {\em Soviet Math. Dokl.}, 1:717--720, 1960.

\bibitem[Nov62]{Novikov62}
S.~P. Novikov.
\newblock Homotopy properties of {T}hom complexes.
\newblock {\em Mat. Sb. (N.S.)}, 57 (99):407--442, 1962.

\bibitem[Nov67]{Novikov67}
S.~P. Novikov.
\newblock Methods of algebraic topology from the point of view of cobordism
  theory.
\newblock {\em Izv. Akad. Nauk SSSR Ser. Mat.}, 31:855--951, 1967.

\bibitem[Pet17]{PetersonChromatic}
Eric Peterson.
\newblock Formal geometry and bordism operations.
\newblock {\em \newline {A}vailable at http://www.math.harvard.edu/~ecp/},
  2017.

\bibitem[Qui69]{Quillen69}
Daniel Quillen.
\newblock On the formal group laws of unoriented and complex cobordism theory.
\newblock {\em Bull. Amer. Math. Soc.}, 75:1293--1298, 1969.

\bibitem[Rav78]{RavenelOdd}
Douglas~C. Ravenel.
\newblock The non-existence of odd primary {A}rf invariant elements in stable
  homotopy.
\newblock {\em Math. Proc. Cambridge Philos. Soc.}, 83(3):429--443, 1978.

\bibitem[Rav92]{RavenelOrangeBook}
Douglas~C. Ravenel.
\newblock {\em Nilpotence and periodicity in stable homotopy theory}, volume
  128 of {\em Annals of Mathematics Studies}.
\newblock Princeton University Press, Princeton, NJ, 1992.
\newblock Appendix C by Jeff Smith.

\bibitem[Rav03]{RavenelGreen}
Douglas~C Ravenel.
\newblock {\em Complex cobordism and stable homotopy groups of spheres}.
\newblock American Mathematical Soc., 2003.

\bibitem[Rez98]{HopkinsMiller}
Charles Rezk.
\newblock Notes on the {H}opkins-{M}iller theorem.
\newblock In {\em Homotopy theory via algebraic geometry and group
  representations ({E}vanston, {IL}, 1997)}, volume 220 of {\em Contemp.
  Math.}, pages 313--366. Amer. Math. Soc., Providence, RI, 1998.

\bibitem[Rob89]{RobinsonAinfty}
Alan Robinson.
\newblock Obstruction theory and the strict associativity of {M}orava
  {$K$}-theories.
\newblock In {\em Advances in homotopy theory ({C}ortona, 1988)}, volume 139 of
  {\em London Math. Soc. Lecture Note Ser.}, pages 143--152. Cambridge Univ.
  Press, Cambridge, 1989.

\end{thebibliography}
\end{document}